\documentclass[12pt]{amsart}
\usepackage{graphicx}
\usepackage{amsmath, amscd, amssymb, amsthm}
\usepackage[subnum]{cases}
\usepackage{changepage}
\usepackage{xcolor}
\usepackage{marginnote}
\usepackage{hyperref}
\usepackage{cleveref}
\usepackage[all]{xy}
\usepackage{tabularx}
\usepackage{ltablex}
\usepackage{longtable}
\usepackage[T1]{fontenc}

\newtheorem{theorem}{Theorem}[section]

\newtheorem{proposition}[theorem]{Proposition}
\newtheorem{corollary}[theorem]{Corollary}
\newtheorem{conjecture}[theorem]{Conjecture}

\theoremstyle{definition}
\newtheorem{definition}[theorem]{Definition}
\newtheorem{remark}[theorem]{Remark}
\numberwithin{equation}{section}

\newtheorem{setting}[theorem]{Setting}

\usepackage{fancyhdr}
\begin{document}

\normalfont

\title{Period Rings with Big Coefficients and Applications I}
\author{Xin Tong}

\maketitle

\begin{abstract}
\rm Following ideas of Kedlaya-Liu, we are going to consider extending our previous work to the context of more general adic spaces, which will be corresponding deformation of the relative $p$-adic Hodge structure over more general adic spaces. This means that the deformation could be also realized by an adic spaces (perfectoid, preperfectoid, relatively perfectoid and so on). Parts of the whole project here actually are inspired by the corresponding Drinfeld's lemma for diamonds after Scholze, as well as the work from Carter-Kedlaya-Z\'abr\'adi which is aimed at studying the representation theory of products of \'etale fundamental groups. Moreover, we gain motivations from noncommutative analytic geometries and noncommutative Tamagawa number conjectures.  
\end{abstract}

\newpage

\tableofcontents

\newpage

\section{Introduction}

\subsection{Some Motivation of Big Coefficients}

\noindent We initiate our development on the corresponding period rings with big coefficients and application, with the corresponding motivation partially coming from our Hodge-Iwasawa theory (see \cite{T1}, \cite{T2}), although we do not want to restrict ourselves to Hodge-Iwasawa theory. In our previous work \cite{T1}, \cite{T2} and actually also \cite{T3}, we developed very carefully the corresponding deformation of various type of relative $p$-adic Hodge structures over rigid analytic spaces. It is not necessary to just consider things over rigid analytic spaces namely those analytic spaces which admit atlas made up of strictly affinoids. We first want to mention that one could have the chance to generalize this.\\

\indent One of our target here is the corresponding adic space coefficient consideration, which is natural if one looks at the corresponding context of \cite{CKZ} or \cite{PZ}, where one has some multivariate Robba rings which admit multi partial Frobenius actions. For instance when we have two variables taking quotient by one partial Frobenius will basically give us some adic space coefficients. The remaining Frobenius gives rise to some relative Frobenius Hodge structures. \\

\indent The corresponding \cite{CKZ} actually proposed that one can also even consider the product version of the corresponding context of \cite{KL1} and \cite{KL2}, which is our second goal here. We want to systematically consider deforming the Frobenius modules structures of \cite{KL1} and \cite{KL2} after our work \cite{T1} and \cite{T2}. We focus on the translating the corresponding results in \cite{T1} and \cite{T2} to the corresponding context in the current adic space consideration, again after \cite{KL1}, \cite{KL2} as well. We hope to work in full generality, namely over perfectoid uniform adic Banach coefficients and those adic space coefficients forming from the quotient of the perfectoid uniform adic Banach ones. \\

\indent In the situation where we have the corresponding Kedlaya-Liu's perfect Robba rings with general coefficients in some Banach rings, one might have to impose some of the corresponding sheafiness condition. But we would like to use the $\infty$-Huber spectrum given by \cite{BK} to tackle this from some other derived perspective where we do have the corresponding derived sheafiness. Definitely this is very complicated to manipulate, but we initiate some of the discussion though not complete at all. We hope we could come back to this systematically later at some point.\\

\indent Finally we revisit the corresponding noncommutative Hodge-Iwasawa theory we initiated in the corresponding papers \cite{T1} and \cite{T2}. We use some observation coming from Kedlaya to tackle the corresponding noncommutative descent for finite projective datum.\\

\subsection{Results}

\indent In this section we make some summary on our main results covered in the main body of our current paper:\\

\noindent 1. The first scope of the discussion is around some perfectoid or preperfectoid deformation of essentially the Frobenius modules and the corresponding quasi-coherent sheaves over the deformed version of the adic Fargues-Fontaine curves by perfectoids or preperfectoids. In \cref{theorem3.11}, we consider the comparison between the following objects (with the corresponding notations in \cref{theorem3.11}):\\

\noindent A. The pseudocoherent sheaves over the adic relative Fargues-Fontaine curve $Y_{\mathrm{FF},R,A}$ where $A$ is a perfectoid adic Banach uniform algebra over $E$, here we consider \'etale topology;\\
\noindent B. The pseudocoherent sheaves over the adic relative Fargues-Fontaine curve $Y_{\mathrm{FF},R,A}$ where $A$ is a perfectoid adic Banach uniform algebra over $E$, here we consider pro-\'etale topology;\\
\noindent C. The pseudocoherent modules over the relative $A$-coefficient Robba ring $\widetilde{\Pi}_{R,A}$, carring the Frobenius action, where $A$ is a perfectoid adic Banach uniform algebra over $E$;\\
\noindent D. The pseudocoherent modules over the relative $A$-coefficient Robba ring $\widetilde{\Pi}^\infty_{R,A}$, carring the Frobenius action, where $A$ is a perfectoid adic Banach uniform algebra over $E$;\\ 
\noindent E. The pseudocoherent bundles over the relative $A$-coefficient Robba ring $\widetilde{\Pi}_{R,A}$, carring the Frobenius action, where $A$ is a perfectoid adic Banach uniform algebra over $E$;\\
\noindent F. The pseudocoherent bundles over the relative $A$-coefficient Robba ring $\widetilde{\Pi}^{[s,r]}_{R,A}$, carring the Frobenius action, where $A$ is a perfectoid adic Banach uniform algebra over $E$, $0<s\leq r/p^{ah}$.

\begin{theorem}\mbox{\bf{(After Kedlaya-Liu \cite[Theorem 4.6.1]{KL2})}}
The categories mentioned above are equivalent to each other. (See \cref{theorem3.11}.)
	
\end{theorem}

\indent Now for an adic space $D$ in the pro-\'etale topology which could be covered by perfectoid subdomains over $E$, we have the following categories (objects defined over $D$ are actually compatible families):\\

\noindent A. The pseudocoherent sheaves over the adic relative Fargues-Fontaine curve $Y_{\mathrm{FF},R,D}$, here we consider \'etale topology;\\
\noindent B. The pseudocoherent sheaves over the adic relative Fargues-Fontaine curve $Y_{\mathrm{FF},R,D}$, here we consider pro-\'etale topology;\\
\noindent C. The pseudocoherent modules over the relative $D$-coefficient Robba ring $\widetilde{\Pi}_{R,D}$, carring the Frobenius action;\\
\noindent D. The pseudocoherent modules over the relative $D$-coefficient Robba ring $\widetilde{\Pi}^\infty_{R,D}$, carring the Frobenius action;\\ 
\noindent E. The pseudocoherent bundles over the relative $D$-coefficient Robba ring $\widetilde{\Pi}_{R,D}$, carring the Frobenius action;\\
\noindent F. The pseudocoherent bundles over the relative $D$-coefficient Robba ring $\widetilde{\Pi}^{[s,r]}_{R,D}$, carring the Frobenius action, $0<s\leq r/p^{ah}$.

\begin{theorem}\mbox{\bf{(After Kedlaya-Liu \cite[Theorem 4.6.1]{KL2})}}
The categories mentioned above are equivalent to each other. (See \cref{theorem3.10}.)
	
\end{theorem}

\indent We then apply to the corresponding construction to the corresponding Fargues-Fontaine curve $Y_{\mathrm{FF},R}$ we have the following categories:\\

\noindent A. The pseudocoherent sheaves over the adic relative Fargues-Fontaine curve $Y_{\mathrm{FF},R,Y_{\mathrm{FF},R}}$, here we consider \'etale topology;\\
\noindent B. The pseudocoherent sheaves over the adic relative Fargues-Fontaine curve $Y_{\mathrm{FF},R,Y_{\mathrm{FF},R}}$, here we consider pro-\'etale topology;\\
\noindent C. The pseudocoherent modules over the relative $Y_{\mathrm{FF},R}$-coefficient Robba ring $\widetilde{\Pi}_{R,Y_{\mathrm{FF},R}}$, carring the Frobenius action;\\
\noindent D. The pseudocoherent modules over the relative $Y_{\mathrm{FF},R}$-coefficient Robba ring $\widetilde{\Pi}^\infty_{R,Y_{\mathrm{FF},R}}$, carring the Frobenius action;\\ 
\noindent E. The pseudocoherent bundles over the relative $Y_{\mathrm{FF},R}$-coefficient Robba ring $\widetilde{\Pi}_{R,Y_{\mathrm{FF},R}}$, carring the Frobenius action;\\
\noindent F. The pseudocoherent bundles over the relative $Y_{\mathrm{FF},R}$-coefficient Robba ring $\widetilde{\Pi}^{[s,r]}_{R,Y_{\mathrm{FF},R}}$, carring the Frobenius action, $0<s\leq r/p^{ah}$.

\begin{theorem}\mbox{\bf{(After Kedlaya-Liu \cite[Theorem 4.6.1]{KL2})}}
The categories mentioned above are equivalent to each other. (See \cref{corollary3.12}.)\\
	
\end{theorem}

\noindent 2. For general Banach objects, we believe the corresponding picture will still be very interesting and have more potential applications. At least the relative analytic geometry (over the period rings) will be very interesting. That being said, the corresponding study of such relative geometry is really not that easy. The first difficulty will be essentially the sheafiness. Bambozzi-Kremnizer established some derived sheafiness when we replace the Huber's spectrum with the derived ones in \cite{BK}. For any Banach ring $S$ we have (with the notation in the \cref{section4}) the $\infty$-analytic space $\mathrm{Spa}^h(S)$. We then apply to our period rings with general Banach coefficient $A$ we have the corresponding $\infty$-period rings. And we then look at the following categories (with the notation in the \cref{section4}):\\

\noindent A. The corresponding $f$-projective Frobenius $\varphi^a$-bundles over the $\infty$-Robba ring $\widetilde{\Pi}^{h}_{R,A}$, here $A$ is arbitrary Banach adic uniform algebra over $\mathbb{Q}_p$;\\
\noindent B. The corresponding $f$-projective Frobenius $\varphi^a$-modules over the $\infty$-Robba ring $\widetilde{\Pi}^{[s,r],h}_{R,A}$, here $0<s\leq r/p^{ah}$.

\begin{theorem}\mbox{\bf{(After Kedlaya-Liu \cite[Theorem 4.6.1]{KL2})}}
The categories mentioned above are equivalent to each other. (See \cref{theorem4.12}.)\\
	
\end{theorem}

\noindent 3. We revisit the corresponding noncommutative descent which was initiated in \cite{T2}. We will noncommutativize some argument due to Kedlaya on the corresponding Kiehl's glueing properties without the noetherian assumption. This means that we could actually descent the corresponding finite projective modules in the corresponding noncommutative setting.

\begin{proposition} 
The descent for finite projective bimodules over noncommutative Banach rings holds under the conditions in \cite[Definition 2.7.3 (a),(b)]{KL1}.	(See \cref{proposition5.11}.)\\
\end{proposition}

\indent Then we apply this to $p$-adic Hodge theory, we look at the following categories:\\
\noindent A. The corresponding finite projective Frobenius $\varphi^a$-bundles over the Robba ring $\widetilde{\Pi}_{R,A}$, here $A$ is arbitrary Banach algebra over $\mathbb{Q}_p$;\\
\noindent B. The corresponding finite projective Frobenius $\varphi^a$-modules over the Robba ring $\widetilde{\Pi}^{[s,r]}_{R,A}$, here $0<s\leq r/p^{ah}$.

\begin{proposition}\mbox{\bf{(After Kedlaya-Liu \cite[Theorem 4.6.1]{KL2})}}
The categories mentioned above are equivalent to each other. (See \cref{proposition5.13}.)\\
\end{proposition}

\subsection{Convention and Notation}

\indent We make some convention here. When we say an $\infty$-analytic stack, we will mean a $\infty$-sheaf fibered over the corresponding category of all the seminormed commutative monoids satisfying the condition and framework in \cite{BBBK}. Please note the difference between a derived analytic stack, an $\infty$-analytic stack and a stack, fibered over category of seminormed commutative monoids, as in \cite{BBBK}.\\

\subsection{Remarks on Future Work}

\indent We will consider more general $\infty$-glueing in the style of Kedlaya-Liu \cite{KL2} after the foundations we have applied from \cite{BK} (or possibly equivalently work of Clausen-Scholze \cite{CS}) whenver one would study more general quasi-coherent sheaves over some sheaves of Banach $\mathcal{O}$-algebras over reasonable sites essentially encoded in our current work.\\

\indent We have not considered the corresponding imperfect setting of the corresponding period rings with really bit coefficients along the corresponding style and fashions we established in \cite{T1} and \cite{T2} in the context of reasonable towers in \cite{KL1} and \cite{KL2}. We would like to study the corresponding commutative and noncommutative setting in future work along these towers going upside and down. There will be some consequences for some specific analytic spaces. Also we could basically contact the corresponding $\infty$-context as well by considering the corresponding sheaves of $\infty$-Banach algebras over some interesting pro-\'etale site over adic spaces.\\

\newpage

\section{Period Rings with Coefficients in General Adic Spaces}

\subsection{Period Rings with Coefficients in General Banach Adic Rings}

\begin{setting}
We now consider the corresponding Kedlaya-Liu's period sheaves with the big coefficients in general adic spaces. We will consider those adic spaces which could be written as the corresponding quotients of the corresponding affinoid (pre-)perfectoid spaces. To be more precise we look at the following rings, we first consider some finite multi-intervals $[s_I,r_I]$ (with real bounds) with $0\leq s_\alpha \leq r_\alpha< \infty$ for each $\alpha \in I$ and we consider the corresponding base field $E$ which is complete nonarchimedian discrete valued carrying some Banach norm, taking the form of $\mathbb{Q}_p$ or $\mathbb{F}_p((\overline{t}))$ with the uniform notation $\pi$ for chosen uniformizer. We then consider over $E$ the corresponding $E$-strictly affinoids:
\begin{align}
E\{s_1/T_1,s_2/T_2,...,s_n/T_n,T_1/r_1,...,T_n/r_n\}	
\end{align}
and we have the corresponding Fr\'echet rings:
\begin{align}
\varprojlim_{r_\alpha \rightarrow 1,\forall \alpha\in I} E\{s_1/T_1,s_2/T_2,...,s_n/T_n,T_1/r_1,...,T_n/r_n\}
\end{align}
with the corresponding ind-Fr\'echet rings:
\begin{align}
\varinjlim_{s_\alpha \rightarrow 0^+s,\forall \alpha\in I}\varprojlim_{r_\alpha \rightarrow 1,\forall \alpha\in I} E\{s_1/T_1,s_2/T_2,...,s_n/T_n,T_1/r_1,...,T_n/r_n\}.
\end{align}
We then consider the following:
\begin{align}
A_{[s_I,r_I]}=E(\pi^{1/p^\infty})\{(s_1/T_1)^{1/p^\infty},(s_2/T_2)^{1/p^\infty},...,(s_n/T_n)^{1/p^\infty},(T_1/r_1)^{1/p^\infty},...,(T_n/r_n)^{1/p^\infty}\}^\wedge,\\
A_{[s_I,r_I]}^+=\mathcal{O}[\pi^{1/p^\infty}]\{(s_1/T_1)^{1/p^\infty},(s_2/T_2)^{1/p^\infty},...,(s_n/T_n)^{1/p^\infty},(T_1/r_1)^{1/p^\infty},...,(T_n/r_n)^{1/p^\infty}\}^\wedge.	
\end{align}
And we also consider the corresponding Fr\'echet ones:
\begin{align}
A_{s_I}= \varprojlim_{r_\alpha \rightarrow 1,\forall \alpha\in I}E(\pi^{1/p^\infty})\{(s_1/T_1)^{1/p^\infty},(s_2/T_2)^{1/p^\infty},...,(s_n/T_n)^{1/p^\infty},(T_1/r_1)^{1/p^\infty},...,(T_n/r_n)^{1/p^\infty}\}^\wedge,\\
A_{r_I}^+= \varprojlim_{r_\alpha \rightarrow 1,\forall \alpha\in I}\mathcal{O}[\pi^{1/p^\infty}]\{(s_1/T_1)^{1/p^\infty},(s_2/T_2)^{1/p^\infty},...,(s_n/T_n)^{1/p^\infty},(T_1/r_1)^{1/p^\infty},...,(T_n/r_n)^{1/p^\infty}\}^\wedge.	
\end{align}
\end{setting}

\begin{setting}
We use the corresponding notations as in the following to denote the corresponding preperfectoid ones:
\begin{align}
A'_{[s_I,r_I]}=E\{(s_1/T_1)^{1/p^\infty},(s_2/T_2)^{1/p^\infty},...,(s_n/T_n)^{1/p^\infty},(T_1/r_1)^{1/p^\infty},...,(T_n/r_n)^{1/p^\infty}\}^\wedge,\\
A_{[s_I,r_I]}^{'+}=\mathcal{O}\{(s_1/T_1)^{1/p^\infty},(s_2/T_2)^{1/p^\infty},...,(s_n/T_n)^{1/p^\infty},(T_1/r_1)^{1/p^\infty},...,(T_n/r_n)^{1/p^\infty}\}^\wedge.	
\end{align}
And we also consider the corresponding Fr\'echet ones:
\begin{align}
A'_{s_I}= \varprojlim_{r_\alpha \rightarrow 1,\forall \alpha\in I}E\{(s_1/T_1)^{1/p^\infty},(s_2/T_2)^{1/p^\infty},...,(s_n/T_n)^{1/p^\infty},(T_1/r_1)^{1/p^\infty},...,(T_n/r_n)^{1/p^\infty}\}^\wedge,\\
A_{r_I}^{'+}= \varprojlim_{r_\alpha \rightarrow 1,\forall \alpha\in I}\mathcal{O}\{(s_1/T_1)^{1/p^\infty},(s_2/T_2)^{1/p^\infty},...,(s_n/T_n)^{1/p^\infty},(T_1/r_1)^{1/p^\infty},...,(T_n/r_n)^{1/p^\infty}\}^\wedge.	
\end{align}	
\end{setting}

\indent Then we consider the corresponding context of \cite[Chapter 1]{T2}, where with trivial coefficients we just consider the same context after \cite{KL2}. First we have the corresponding consideration, let $F$ be the corresponding field which is a complete discrete valued nnonarchimedean field with uniformizer $\pi$ with normalized Banach norm $\|.\|_F$ such the corresponding evaluation on $\pi$ is just $p^{-1}$, and we assume the residue field of $F$ is just $\mathbb{F}_{p^h}$ for some integer $h>0$. Now we define the following big Robba rings:

\begin{definition} \mbox{\bf{(After Kedlaya-Liu \cite[Definition 4.1.1]{KL2})}}
We let $A$ be a perfectoid adic uniform Banach algebra over $E$ with integral subring $\mathcal{O}_A$ over $\mathcal{O}_E$. Recall from the corresponding context in \cite{KL1} we have the corresponding period rings in the relative setting. We follow the corresponding notations we used in \cite[Section 2.1]{T2} for those corresponding period rings. We first have for a pair $(R,R^+)$ where $R$ is a uniform perfect adic Banach ring over the assumed base $\mathcal{O}_F$. Then we take the corresponding generalized Witt vectors taking the form of $W_{\mathcal{O}_E}(R)$, which is just the ring $\widetilde{\Omega}_R^\mathrm{int}$, and by inverting the corresponding uniformizer we have the ring $\widetilde{\Omega}_R$, then by taking the completed product with $A$ we have $\widetilde{\Omega}_{R,A}$. Now for some $r>0$ we consider the ring $\widetilde{\Pi}^{\mathrm{int},r}_{R}$ which is the completion of $W_{\mathcal{O}_E}(R^+)[[R]]$ (note that this is actually just $W_{\mathcal{O}_E}(R^+)[[x],x\in R]$) by the norm $\|.\|_{\alpha^r}$ defined by:
\begin{align}
\|.\|_{\alpha^r}(\sum_{n\geq 0}\pi^n[\overline{r}_n])=\sup_{n\geq 0}\{p^{-n}\alpha(\overline{r}_n)^r\}.	
\end{align}
Then we have the product $\widetilde{\Pi}^{\mathrm{bd},r}_{R,A}$ defined as the completion under the product norm $\|.\|_{\alpha^r}\otimes \|.\|_A$ of the corresponding ring $\widetilde{\Pi}^{\mathrm{bd},r}_{R}\otimes_{E}A$, which could be also defined from the corresponding integral Robba rings defined above. Then we define the corresponding Robba ring for some interval $I\subset (0,\infty)$ with coefficient in the perfectoid ring $A$ denoted by $\widetilde{\Pi}^{I}_{R,A}$ as the following complete tensor product:
\begin{displaymath}
\widetilde{\Pi}^{I}_{R}\widehat{\otimes}_{E} A	
\end{displaymath}
under the the corresponding tensor product norm $\|.\|_{\alpha^r}\otimes \|.\|_A$. Then we set $\widetilde{\Pi}^r_{R,A}$ as $\varprojlim_{s\rightarrow 0}\widetilde{\Pi}^{[s,r]}_{R,A}$, and then we define $\widetilde{\Pi}_{R,A}$ as $\varinjlim_{r\rightarrow \infty}\widetilde{\Pi}^{[s,r]}_{R,A}$, and we define $\widetilde{\Pi}^\infty_{R,A}$ as $\varprojlim_{r\rightarrow \infty}\widetilde{\Pi}^{r}_{R,A}$. And we also have the corresponding full integral Robba ring and the corresponding full bounded Robba ring by taking the corresponding union through all $r>0$.
\end{definition}

\indent One can even consider the corresponding coefficients in the corresponding Banach adic uniform algebra over the base field.

\begin{definition} \mbox{\bf{(After Kedlaya-Liu \cite[Definition 4.1.1]{KL2})}}
We let $B$ be an adic uniform Banach algebra over $E$ with integral subring $\mathcal{O}_B$ over $\mathcal{O}_E$. Recall from the corresponding context in \cite{KL1} we have the corresponding period rings in the relative setting. We follow the corresponding notations we used in \cite[Section 2.1]{T2} for those corresponding period rings. We first have for a pair $(R,R^+)$ where $R$ is a uniform perfect adic Banach ring over the assumed base $\mathcal{O}_F$. Then we take the corresponding generalized Witt vectors taking the form of $W_{\mathcal{O}_E}(R)$, which is just the ring $\widetilde{\Omega}_R^\mathrm{int}$, and by inverting the corresponding uniformizer we have the ring $\widetilde{\Omega}_R$, then by taking the completed product with $B$ we have $\widetilde{\Omega}_{R,B}$. Now for some $r>0$ we consider the ring $\widetilde{\Pi}^{\mathrm{int},r}_{R}$ which is the completion of $W_{\mathcal{O}_E}(R^+)[[r]:r\in R] $ by the norm $\|.\|_{\alpha^r}$ defined by:
\begin{align}
\|.\|_{\alpha^r}(\sum_{n\geq 0}\pi^n[\overline{r}_n])=\sup_{n\geq 0}\{p^{-n}\alpha(\overline{r}_n)^r\}.	
\end{align}
Then we have the product $\widetilde{\Pi}^{\mathrm{bd},r}_{R,B}$ defined as the completion under the product norm $\|.\|_{\alpha^r}\otimes \|.\|_B$ of the corresponding ring $\widetilde{\Pi}^{\mathrm{bd},r}_{R}\otimes_{E }B$, which could be also defined from the corresponding integral Robba rings defined above. Then we define the corresponding Robba ring for some interval $I\subset (0,\infty)$ with coefficient in the perfectoid ring $B$ denoted by $\widetilde{\Pi}^{I}_{R,B}$ as the following complete tensor product:
\begin{displaymath}
\widetilde{\Pi}^{I}_{R}\widehat{\otimes}_{E} B	
\end{displaymath}
under the the corresponding tensor product norm $\|.\|_{\alpha^r}\otimes \|.\|_B$. Then we set $\widetilde{\Pi}^r_{R,B}$ as $\varprojlim_{s\rightarrow 0}\widetilde{\Pi}^{[s,r]}_{R,B}$, and then we define $\widetilde{\Pi}_{R,B}$ as $\varinjlim_{r\rightarrow \infty}\widetilde{\Pi}^{[s,r]}_{R,B}$, and we define $\widetilde{\Pi}^\infty_{R,B}$ as $\varprojlim_{r\rightarrow \infty}\widetilde{\Pi}^{r}_{R,B}$. And we also have the corresponding full integral Robba ring and the corresponding full bounded Robba ring by taking the corresponding union through all $r>0$.
\end{definition}

\indent In some situation we also need the corresponding preperfectoid coefficients for instance the local chart coming from the adic Fargues-Fontaine curves:

\begin{definition}  \mbox{\bf{(After Kedlaya-Liu \cite[Definition 4.1.1]{KL2})}}
We let $A$ be a preperfectoid adic uniform Banach algebra over $E$ with integral subring $\mathcal{O}_A$ over $\mathcal{O}_E$. Recall from the corresponding context in \cite{KL1} we have the corresponding period rings in the relative setting. We follow the corresponding notations we used in \cite[Section 2.1]{T2} for those corresponding period rings. We first have for a pair $(R,R^+)$ where $R$ is a uniform perfect adic Banach ring over the assumed base $\mathcal{O}_F$. Then we take the corresponding generalized Witt vectors taking the form of $W_{\mathcal{O}_E}(R)$, which is just the ring $\widetilde{\Omega}_R^\mathrm{int}$, and by inverting the corresponding uniformizer we have the ring $\widetilde{\Omega}_R$, then by taking the completed product with $A$ we have $\widetilde{\Omega}_{R,A}$. Now for some $r>0$ we consider the ring $\widetilde{\Pi}^{\mathrm{int},r}_{R}$ which is the completion of $W_{\mathcal{O}_E}(R^+)[[r]:r\in R] $ by the norm $\|.\|_{\alpha^r}$ defined by:
\begin{align}
\|.\|_{\alpha^r}(\sum_{n\geq 0}\pi^n[\overline{r}_n])=\sup_{n\geq 0}\{p^{-n}\alpha(\overline{r}_n)^r\}.	
\end{align}
Then we have the product $\widetilde{\Pi}^{\mathrm{bd},r}_{R,A}$ defined as the completion under the product norm $\|.\|_{\alpha^r}\otimes \|.\|_A$ of the corresponding ring $\widetilde{\Pi}^{\mathrm{bd},r}_{R}\otimes_{E}A$, which could be also defined from the corresponding integral Robba rings defined above. Then we define the corresponding Robba ring for some interval $I\subset (0,\infty)$ with coefficient in the perfectoid ring $A$ denoted by $\widetilde{\Pi}^{I}_{R,A}$ as the following complete tensor product:
\begin{displaymath}
\widetilde{\Pi}^{I}_{R}\widehat{\otimes}_{E} A	
\end{displaymath}
under the the corresponding tensor product norm $\|.\|_{\alpha^r}\otimes \|.\|_A$. Then we set $\widetilde{\Pi}^r_{R,A}$ as $\varprojlim_{s\rightarrow 0}\widetilde{\Pi}^{[s,r]}_{R,A}$, and then we define $\widetilde{\Pi}_{R,A}$ as $\varinjlim_{r\rightarrow \infty}\widetilde{\Pi}^{[s,r]}_{R,A}$, and we define $\widetilde{\Pi}^\infty_{R,A}$ as $\varprojlim_{r\rightarrow \infty}\widetilde{\Pi}^{r}_{R,A}$. And we also have the corresponding full integral Robba ring and the corresponding full bounded Robba ring by taking the corresponding union through all $r>0$.
\end{definition}

\indent In the previous situation where we have the more explicit representation of the corresponding perfectoid coefficients we can basically consider more explicit construction:

\begin{definition}\mbox{\bf{(After Kedlaya-Liu \cite[Definition 4.1.1]{KL2})}}
We let $A$ be a perfectoid adic uniform Banach algebra over $E$ with integral subring $\mathcal{O}_A$ over $\mathcal{O}_E$ taking the form of one of the rings $A_{[s_I,r_I]}$. Recall from the corresponding context in \cite{KL1} we have the corresponding period rings in the relative setting. We follow the corresponding notations we used in \cite[Section 2.1]{T2} for those corresponding period rings. We first have for a pair $(R,R^+)$ where $R$ is a uniform perfect adic Banach ring over the assumed base $\mathcal{O}_F$. Then we take the corresponding generalized Witt vectors taking the form of $W_{\mathcal{O}_E}(R)$, which is just the ring $\widetilde{\Omega}_R^\mathrm{int}$, and by inverting the corresponding uniformizer we have the ring $\widetilde{\Omega}_R$, then by taking the completed product with $A$ we have $\widetilde{\Omega}_{R,A}$. Now for some $r>0$ we consider the ring $\widetilde{\Pi}^{\mathrm{int},r}_{R,A}$ which is the completion of $W_{\mathcal{O}_E}(R^+)[[r]:r\in R]\otimes \mathcal{O}_A $ by the norm $\|.\|_{\alpha^r,A}$ defined by:
\begin{align}
\|.\|_{\alpha^r}(\sum_{n,i_1,...,i_k,j_1,...,j_k\in \mathbb{Z}[1/p]_{\geq 0}}\pi^n[\overline{r}_n](s_1/T_1)^{i_1}...(s_k/T_k)^{i_k}(T_1/r_1)^{j_1}...(T_k/r_k)^{j_k})=\sup_{n\geq 0}\{p^{-n}\alpha(\overline{r}_n)^r\}.	
\end{align}
Then we have the product $\widetilde{\Pi}^{\mathrm{bd},r}_{R,A}$ defined as the completion under the product norm $\|.\|_{\alpha^r}\otimes \|.\|_A$ of the corresponding ring $\widetilde{\Pi}^{\mathrm{bd},r}_{R}\otimes_{E}A$, which could be also defined from the corresponding integral Robba rings defined above. Then we define the corresponding Robba ring for some interval $I\subset (0,\infty)$ with coefficient in the perfectoid ring $A$ denoted by $\widetilde{\Pi}^{I}_{R,A}$ as the following complete tensor product:
\begin{displaymath}
\widetilde{\Pi}^{I}_{R}\widehat{\otimes}_{E} E(\pi^{1/p^\infty})\{(s_1/T_1)^{1/p^\infty},(s_2/T_2)^{1/p^\infty},...,(s_n/T_n)^{1/p^\infty},(T_1/r_1)^{1/p^\infty},...,(T_n/r_n)^{1/p^\infty}\}^\wedge	
\end{displaymath}
under the the corresponding tensor product norm $\|.\|_{\alpha^r}\otimes \|.\|_A$. Then we set $\widetilde{\Pi}^r_{R,A}$ as $\varprojlim_{s\rightarrow 0}\widetilde{\Pi}^{[s,r]}_{R,A}$, and then we define $\widetilde{\Pi}_{R,A}$ as $\varinjlim_{r\rightarrow \infty}\widetilde{\Pi}^{[s,r]}_{R,A}$, and we define $\widetilde{\Pi}^\infty_{R,A}$ as $\varprojlim_{r\rightarrow \infty}\widetilde{\Pi}^{r}_{R,A}$. And we also have the corresponding full integral Robba ring and the corresponding full bounded Robba ring by taking the corresponding union through all $r>0$.\\
\end{definition}


\subsection{Properties}

\indent As in our work \cite{T2} and \cite{KL2}, we can discuss some of the key properties of the corresponding period rings defined above.

\begin{setting}
We will mainly consider the corresponding coefficients in 
\begin{center}
$E(\pi^{1/p^\infty})\{(s_1/T_1)^{1/p^\infty},(s_2/T_2)^{1/p^\infty},...,(s_n/T_n)^{1/p^\infty},(T_1/r_1)^{1/p^\infty},...,(T_n/r_n)^{1/p^\infty}\}^\wedge$ 
\end{center}
and their Banach uniform adic strict quotient through the following strict morphism:
\begin{displaymath}
A_{[s_I,r_I]}\rightarrow \overline{A_{[s_I,r_I]}}\rightarrow 0.	
\end{displaymath}
Also we consider the corresponding preperfectoid setting:
\begin{center}
$E\{(s_1/T_1)^{1/p^\infty},(s_2/T_2)^{1/p^\infty},...,(s_n/T_n)^{1/p^\infty},(T_1/r_1)^{1/p^\infty},...,(T_n/r_n)^{1/p^\infty}\}^\wedge$ 
\end{center}
and their Banach uniform adic strict quotient through the following strict morphism:
\begin{displaymath}
A'_{[s_I,r_I]}\rightarrow \overline{A'_{[s_I,r_I]}}\rightarrow 0.	
\end{displaymath} 
 	
\end{setting}

	

	

\begin{proposition} \mbox{\bf{(After Kedlaya-Liu \cite[Lemma 5.2.6]{KL1})}} 
Under the corresponding definitions and the corresponding notations in our context we have the following identification:
\begin{align}
\widetilde{\Pi}^{\mathrm{int},s}_{R,A_{[s_I,r_I]}}\bigcap \widetilde{\Pi}^{[s,r]}_{R,A_{[s_I,r_I]}}=	\widetilde{\Pi}^{\mathrm{int},r}_{R,A_{[s_I,r_I]}}\\
\widetilde{\Pi}^{\mathrm{int},s}_{R,\overline{A_{[s_I,r_I]}}}\bigcap \widetilde{\Pi}^{[s,r]}_{R,\overline{A_{[s_I,r_I]}}}=	\widetilde{\Pi}^{\mathrm{int},r}_{R,\overline{A_{[s_I,r_I]}}}.
\end{align}
	
\end{proposition}

\begin{proof}
For the first corresponding identity above we perform the corresponding argument as in \cite[Lemma 5.2.6]{KL2} and \cite{T2}. First consider now arbitrary element $h$ in the corresponding intersection on the left. Then we use a corresponding sequence of elements in the bounded Robba ring to approximate the corresponding element $h$ above, say $h_1,h_2,...$. This will mean that for any $k$ we could find some sufficiently large number $P_k$ such that the following holds for any $t\in [s,r]$ and for any $i\geq P_k$:
\begin{align}
\|.\|_{\alpha^t,A_{[s_I,r_I]}}(h-h_i)\leq p^{-k}.	
\end{align}
Then we are going to consider the corresponding decomposition of each element $h_i$ in the following way:
\begin{align}
h_i=\sum_{n\in \mathbb{Z}[1/p],i_1,...,i_k,j_1,...,j_k\in \mathbb{Z}[1/p]_{\geq 0}}\pi^n[\overline{h}_{i,n}](s_1/T_1)^{i_1}...(s_k/T_k)^{i_k}(T_1/r_1)^{j_1}...(T_k/r_k)^{j_k}.	
\end{align}
Then we consider the corresponding integral part of this decomposition namely we consider:
\begin{align}
g_i=\sum_{n\in \mathbb{Z}[1/p]_{\geq 0},i_1,...,i_k,j_1,...,j_k\in \mathbb{Z}[1/p]_{\geq 0}}\pi^n[\overline{h}_{i,n}](s_1/T_1)^{i_1}...(s_k/T_k)^{i_k}(T_1/r_1)^{j_1}...(T_k/r_k)^{j_k}.	
\end{align}
Now since we have that the elements involved are also living in the integral rings we have then that in the expression of $h_i$ we have:
\begin{align}
\|.\|_{\alpha^s,A_{[s_I,r_I]}}(\pi^n[\overline{h}_{i,n}](s_1/T_1)^{i_1}...(s_k/T_k)^{i_k}(T_1/r_1)^{j_1}...(T_k/r_k)^{j_k})\leq p^{-k},\forall n<0.	
\end{align}
Then we have:
\begin{align}
{\alpha}(\overline{h}_{i,n})\leq p^{-(k-n)/s}.	
\end{align}
Then we have:
\begin{displaymath}
\|.\|_{\alpha^r,A_{[s_I,r_I]}}(h-g_i)\leq p^{-n}p^{-(i-n)r/s}.	
\end{displaymath}
This will prove the corresponding results since we can now use the corresponding elements in the integral Robba rings to approximate.\\
For the corresponding second corresponding identity we consider the corresponding strictness of the quotient in the original formation of the rings. Now take any element $\overline{h}$ in the quotient, then lift this to $h$. For the first corresponding identity above we perform the corresponding argument as in \cite[Lemma 5.2.6]{KL2} and \cite{T2}. First consider now arbitrary element $h$ in the corresponding intersection on the left. Then we use a corresponding sequence of elements in the bounded Robba ring to approximate the corresponding element $h$ above, say $h_1,h_2,...$. This will mean that for any $k$ we could find some sufficiently large number $P_k$ such that the following holds for any $t\in [s,r]$ and for any $i\geq P_k$:
\begin{align}
\|.\|_{\alpha^t,A_{[s_I,r_I]}}(h-h_i)\leq p^{-k}.	
\end{align}
Then we are going to consider the corresponding decomposition of each element $h_i$ in the following way:
\begin{align}
h_i=\sum_{n\in \mathbb{Z}[1/p],i_1,...,i_k,j_1,...,j_k\in \mathbb{Z}[1/p]_{\geq 0}}\pi^n[\overline{h}_{i,n}](s_1/T_1)^{i_1}...(s_k/T_k)^{i_k}(T_1/r_1)^{j_1}...(T_k/r_k)^{j_k}.	
\end{align}
Then we consider the corresponding integral part of this decomposition namely we consider:
\begin{align}
g_i=\sum_{n\in \mathbb{Z}[1/p]_{\geq 0},i_1,...,i_k,j_1,...,j_k\in \mathbb{Z}[1/p]_{\geq 0}}\pi^n[\overline{h}_{i,n}](s_1/T_1)^{i_1}...(s_k/T_k)^{i_k}(T_1/r_1)^{j_1}...(T_k/r_k)^{j_k}.	
\end{align}
Now since we have that the elements involved are also living in the integral rings we have then that in the expression of $h_i$ we have:
\begin{align}
\|.\|_{\alpha^s,A_{[s_I,r_I]}}(\pi^n[\overline{h}_{i,n}](s_1/T_1)^{i_1}...(s_k/T_k)^{i_k}(T_1/r_1)^{j_1}...(T_k/r_k)^{j_k})\leq p^{-k},\forall n<0.	
\end{align}
Then we have:
\begin{align}
{\alpha}(\overline{h}_{i,n})\leq p^{-(k-n)/s}.	
\end{align}
Then we have:
\begin{displaymath}
\|.\|_{\alpha^r,A_{[s_I,r_I]}}(h-g_i)\leq p^{-n}p^{-(i-n)r/s}.	
\end{displaymath}
This will prove the corresponding results since we can now use the corresponding elements in the integral Robba rings to approximate. Then we have that:
\begin{align}
\|.\|_{\alpha^r,\overline{A_{[s_I,r_I]}}}(\overline{h}-\overline{g}_i)\leq p^{-n}p^{-(i-n)r/s},	
\end{align}
which shows the corresponding approximating process in the strict quotient.

\end{proof}

\begin{proposition} \mbox{\bf{(After Kedlaya-Liu \cite[Lemma 5.2.6]{KL1})}} 
Under the corresponding definitions and the corresponding notations in our context we have the following identification:
\begin{align}
\widetilde{\Pi}^{\mathrm{int},s}_{R,A'_{[s_I,r_I]}}\bigcap \widetilde{\Pi}^{[s,r]}_{R,A'_{[s_I,r_I]}}=	\widetilde{\Pi}^{\mathrm{int},r}_{R,A'_{[s_I,r_I]}}\\
\widetilde{\Pi}^{\mathrm{int},s}_{R,\overline{A_{[s_I,r_I]}}}\bigcap \widetilde{\Pi}^{[s,r]}_{R,\overline{A'_{[s_I,r_I]}}}=	\widetilde{\Pi}^{\mathrm{int},r}_{R,\overline{A'_{[s_I,r_I]}}}.
\end{align}
	
\end{proposition}

\begin{proof}
For the first corresponding identity above we perform the corresponding argument as in \cite[Lemma 5.2.6]{KL2} and \cite{T2}. First consider now arbitrary element $h$ in the corresponding intersection on the left. Then we use a corresponding sequence of elements in the bounded Robba ring to approximate the corresponding element $h$ above, say $h_1,h_2,...$. This will mean that for any $k$ we could find some sufficiently large number $P_k$ such that the following holds for any $t\in [s,r]$ and for any $i\geq P_k$:
\begin{align}
\|.\|_{\alpha^t,A'_{[s_I,r_I]}}(h-h_i)\leq p^{-k}.	
\end{align}
Then we are going to consider the corresponding decomposition of each element $h_i$ in the following way:
\begin{align}
h_i=\sum_{n\in \mathbb{Z},i_1,...,i_k,j_1,...,j_k\in \mathbb{Z}[1/p]_{\geq 0}}\pi^n[\overline{h}_{i,n}](s_1/T_1)^{i_1}...(s_k/T_k)^{i_k}(T_1/r_1)^{j_1}...(T_k/r_k)^{j_k}.	
\end{align}
Then we consider the corresponding integral part of this decomposition namely we consider:
\begin{align}
g_i=\sum_{n\in \mathbb{Z}_{\geq 0},i_1,...,i_k,j_1,...,j_k\in \mathbb{Z}[1/p]_{\geq 0}}\pi^n[\overline{h}_{i,n}](s_1/T_1)^{i_1}...(s_k/T_k)^{i_k}(T_1/r_1)^{j_1}...(T_k/r_k)^{j_k}.	
\end{align}
Now since we have that the elements involved are also living in the integral rings we have then that in the expression of $h_i$ we have:
\begin{align}
\|.\|_{\alpha^s,A'_{[s_I,r_I]}}(\pi^n[\overline{h}_{i,n}](s_1/T_1)^{i_1}...(s_k/T_k)^{i_k}(T_1/r_1)^{j_1}...(T_k/r_k)^{j_k})\leq p^{-k},\forall n<0.	
\end{align}
Then we have:
\begin{align}
{\alpha}(\overline{h}_{i,n})\leq p^{-(k-n)/s}.	
\end{align}
Then we have:
\begin{displaymath}
\|.\|_{\alpha^r,A'_{[s_I,r_I]}}(h-g_i)\leq p^{-n}p^{-(i-n)r/s}.	
\end{displaymath}
This will prove the corresponding results since we can now use the corresponding elements in the integral Robba rings to approximate.\\
For the corresponding second corresponding identity we consider the corresponding strictness of the quotient in the original formation of the rings. Now take any element $\overline{h}$ in the quotient, then lift this to $h$. For the first corresponding identity above we perform the corresponding argument as in \cite[Lemma 5.2.6]{KL2} and \cite{T2}. First consider now arbitrary element $h$ in the corresponding intersection on the left. Then we use a corresponding sequence of elements in the bounded Robba ring to approximate the corresponding element $h$ above, say $h_1,h_2,...$. This will mean that for any $k$ we could find some sufficiently large number $P_k$ such that the following holds for any $t\in [s,r]$ and for any $i\geq P_k$:
\begin{align}
\|.\|_{\alpha^t,A'_{[s_I,r_I]}}(h-h_i)\leq p^{-k}.	
\end{align}
Then we are going to consider the corresponding decomposition of each element $h_i$ in the following way:
\begin{align}
h_i=\sum_{n\in \mathbb{Z},i_1,...,i_k,j_1,...,j_k\in \mathbb{Z}[1/p]_{\geq 0}}\pi^n[\overline{h}_{i,n}](s_1/T_1)^{i_1}...(s_k/T_k)^{i_k}(T_1/r_1)^{j_1}...(T_k/r_k)^{j_k}.	
\end{align}
Then we consider the corresponding integral part of this decomposition namely we consider:
\begin{align}
g_i=\sum_{n\in \mathbb{Z}_{\geq 0},i_1,...,i_k,j_1,...,j_k\in \mathbb{Z}[1/p]_{\geq 0}}\pi^n[\overline{h}_{i,n}](s_1/T_1)^{i_1}...(s_k/T_k)^{i_k}(T_1/r_1)^{j_1}...(T_k/r_k)^{j_k}.	
\end{align}
Now since we have that the elements involved are also living in the integral rings we have then that in the expression of $h_i$ we have:
\begin{align}
\|.\|_{\alpha^s,A'_{[s_I,r_I]}}(\pi^n[\overline{h}_{i,n}](s_1/T_1)^{i_1}...(s_k/T_k)^{i_k}(T_1/r_1)^{j_1}...(T_k/r_k)^{j_k})\leq p^{-k},\forall n<0.	
\end{align}
Then we have:
\begin{align}
{\alpha}(\overline{h}_{i,n})\leq p^{-(k-n)/s}.	
\end{align}
Then we have:
\begin{displaymath}
\|.\|_{\alpha^r,A'_{[s_I,r_I]}}(h-g_i)\leq p^{-n}p^{-(i-n)r/s}.	
\end{displaymath}
This will prove the corresponding results since we can now use the corresponding elements in the integral Robba rings to approximate. Then we have that:
\begin{align}
\|.\|_{\alpha^r,\overline{A'_{[s_I,r_I]}}}(\overline{h}-\overline{g}_i)\leq p^{-n}p^{-(i-n)r/s},	
\end{align}
which shows the corresponding approximating process in the strict quotient.

\end{proof}

\begin{proposition} \mbox{\bf{(After Kedlaya-Liu \cite[Lemma 5.2.7]{KL1})}} 
We consider the Robba ring over some closed interval $\widetilde{\Pi}^{[s,r]}_{R,A_{[s_I,r_I]}}$ we have the following decomposition for each element $x$ in this ring into the following way:
\begin{displaymath}
x=y+z	
\end{displaymath}
where $y\in \pi^n \widetilde{\Pi}^{\mathrm{int},r}_{R,A_{[s_I,r_I]}}$ and $z\in \widetilde{\Pi}^{[s,r']}_{R,A_{[s_I,r_I]}}$ for each $r'\geq r$, and here $n>0$ lives in $\mathbb{Z}[1/p]$. With this decomposition we have the following estimate:
\begin{align}
\|.\|_{\alpha^t,A_{[s_I,r_I]}}(z)\leq p^{(1-n)(1-t/r)}\|.\|_{\alpha^r,A_{[s_I,r_I]}}(x)^{t/r}.	
\end{align}
And we consider the Robba ring over some closed interval $\widetilde{\Pi}^{[s,r]}_{R,\overline{A_{[s_I,r_I]}}}$ we have the following decomposition for each element $x$ in this ring into the following way:
\begin{displaymath}
x=y+z	
\end{displaymath}
where $y\in p^n \widetilde{\Pi}^{\mathrm{int},r}_{R,\overline{A_{[s_I,r_I]}}}$ and $z\in \widetilde{\Pi}^{[s,r']}_{R,\overline{A_{[s_I,r_I]}}}$ for each $r'\geq r$, and here $n>0$ lives in $\mathbb{Z}[1/p]$. With this decomposition we have the following estimate:
\begin{align}
\|.\|_{\alpha^t,\overline{A_{[s_I,r_I]}}}(z)\leq p^{(1-n)(1-t/r)}\|.\|_{\alpha^r,\overline{A_{[s_I,r_I]}}}(x)^{t/r}.	
\end{align}
	
\end{proposition}

\begin{proof}
For the first part, we first consider any element $r$ in the bounded Robba ring. Each such element admits the corresponding expansion by the linear combination of the Teichm\"uller elements, each term takes the corresponding form as $\pi^l[r_{l}]$ summing up from some negative integer. Then the corresponding decomposition just reads:
\begin{align}
y=\sum_{l\geq n,i_1,...,i_k,j_1,...,j_k\in \mathbb{Z}[1/p]_{\geq 0}}\pi^l[\overline{r}_{l}](s_1/T_1)^{i_1}...(s_k/T_k)^{i_k}(T_1/r_1)^{j_1}...(T_k/r_k)^{j_k},\\
z=r-y.	
\end{align}
Then in this case we have the corresponding desired estimate:
\begin{align}
\|.\|_{\alpha^t,A_{[s_I,r_I]}}(\pi^l[\overline{r}_{l}](s_1/T_1)^{i_1}...(s_k/T_k)^{i_k}(T_1/r_1)^{j_1}...(T_k/r_k)^{j_k})\\
\leq p^{(1-n)(1-t/r)}\|.\|_{\alpha^r,A_{[s_I,r_I]}}(\pi^l[\overline{r}_{l}](s_1/T_1)^{i_1}...(s_k/T_k)^{i_k}(T_1/r_1)^{j_1}...(T_k/r_k)^{j_k})^{t/r}.
\end{align}
 Then we choose a corresponding sequence of the corresponding elements in the bounded Robba ring to approximate any chosen arbitrary element in $\widetilde{\Pi}^{[s,r]}_{R,A_{[s_I,r_I]}}$ such that we have:
 \begin{displaymath}
 \|.\|_{\alpha^t,A_{[s_I,r_I]}}(r-r_0-...-r_i)\leq p^{-1-i}  \|.\|_{\alpha^t,A_{[s_I,r_I]}}(r),\forall t\in [s,r].	
 \end{displaymath}
Then apply the corresponding construction in the previous situation to each $r_i$ we have that there is a corresponding decomposition:
\begin{displaymath}
r_i=y_i+z_i	
\end{displaymath}
with the corresponding desired extracted sum $\sum_{i}y_i$ which converges to the desired element $y$. While for $z_i$ we consider the corresponding estimate as in the following:
\begin{align}
 \|.\|_{\alpha^t,A_{[s_I,r_I]}}(z_i) &\leq  p^{(1-n)(1-t/r)}\|.\|_{\alpha^r,A_{[s_I,r_I]}}(r_i)^{t/r}\\
 &\leq p^{-i}p^{(1-n)(1-t/r)}\|.\|_{\alpha^r,A_{[s_I,r_I]}}(r_i)^{t/r},	
\end{align}
which shows the corresponding sum $\sum_i z_i$ converges to desired element $z$ in the corresponding union of all the Robba rings $\widetilde{\Pi}^{[s,r']}_{R,A_{[s_I,r_I]}}$ for all $r'\geq r$. Then for the second statement we just perform as before to look at the corresponding lifting of elements through the corresponding strict quotient morphism.
 
\end{proof}

\begin{proposition} \mbox{\bf{(After Kedlaya-Liu \cite[Lemma 5.2.7]{KL1})}} 
We consider the Robba ring over some closed interval $\widetilde{\Pi}^{[s,r]}_{R,A'_{[s_I,r_I]}}$ we have the following decomposition for each element $x$ in this ring into the following way:
\begin{displaymath}
x=y+z	
\end{displaymath}
where $y\in \pi^n \widetilde{\Pi}^{\mathrm{int},r}_{R,A'_{[s_I,r_I]}}$ and $z\in \widetilde{\Pi}^{[s,r']}_{R,A'_{[s_I,r_I]}}$ for each $r'\geq r$, and here $n>0$ is some arbitrary given integer. With this decomposition we have the following estimate:
\begin{align}
\|.\|_{\alpha^t,A'_{[s_I,r_I]}}(z)\leq p^{(1-n)(1-t/r)}\|.\|_{\alpha^r,A'_{[s_I,r_I]}}(x)^{t/r}.	
\end{align}
And we consider the Robba ring over some closed interval $\widetilde{\Pi}^{[s,r]}_{R,\overline{A'_{[s_I,r_I]}}}$ we have the following decomposition for each element $x$ in this ring into the following way:
\begin{displaymath}
x=y+z	
\end{displaymath}
where $y\in p^n \widetilde{\Pi}^{\mathrm{int},r}_{R,\overline{A'_{[s_I,r_I]}}}$ and $z\in \widetilde{\Pi}^{[s,r']}_{R,\overline{A'_{[s_I,r_I]}}}$ for each $r'\geq r$, and here $n>0$ lives in $\mathbb{Z}$. With this decomposition we have the following estimate:
\begin{align}
\|.\|_{\alpha^t,\overline{A'_{[s_I,r_I]}}}(z)\leq p^{(1-n)(1-t/r)}\|.\|_{\alpha^r,\overline{A'_{[s_I,r_I]}}}(x)^{t/r}.	
\end{align}

\end{proposition}

\begin{proof}
For the first part, we first consider any element $r$ in the bounded Robba ring. Each such element admits the corresponding expansion by the linear combination of the Teichm\"uller elements, each term takes the corresponding form as $\pi^l[r_{l}]$ summing up from some negative integer. Then the corresponding decomposition just reads:
\begin{align}
y=\sum_{l\geq n,i_1,...,i_k,j_1,...,j_k\in \mathbb{Z}[1/p]_{\geq 0}}\pi^l[\overline{r}_{l}](s_1/T_1)^{i_1}...(s_k/T_k)^{i_k}(T_1/r_1)^{j_1}...(T_k/r_k)^{j_k},\\
z=r-y.	
\end{align}
Then in this case we have the corresponding desired estimate:
\begin{align}
\|.\|_{\alpha^t,A'_{[s_I,r_I]}}(\pi^l[\overline{r}_{l}](s_1/T_1)^{i_1}...(s_k/T_k)^{i_k}(T_1/r_1)^{j_1}...(T_k/r_k)^{j_k})\\
\leq p^{(1-n)(1-t/r)}\|.\|_{\alpha^r,A'_{[s_I,r_I]}}(\pi^l[\overline{r}_{l}](s_1/T_1)^{i_1}...(s_k/T_k)^{i_k}(T_1/r_1)^{j_1}...(T_k/r_k)^{j_k})^{t/r}.
\end{align}
 Then we choose a corresponding sequence of the corresponding elements in the bounded Robba ring to approximate any chosen arbitrary element in $\widetilde{\Pi}^{[s,r]}_{R,A'_{[s_I,r_I]}}$ such that we have:
 \begin{displaymath}
 \|.\|_{\alpha^t,A'_{[s_I,r_I]}}(r-r_0-...-r_i)\leq p^{-1-i}  \|.\|_{\alpha^t,A'_{[s_I,r_I]}}(r),\forall t\in [s,r].	
 \end{displaymath}
Then apply the corresponding construction in the previous situation to each $r_i$ we have that there is a corresponding decomposition:
\begin{displaymath}
r_i=y_i+z_i	
\end{displaymath}
with the corresponding desired extracted sum $\sum_{i}y_i$ which converges to the desired element $y$. While for $z_i$ we consider the corresponding estimate as in the following:
\begin{align}
 \|.\|_{\alpha^t,A'_{[s_I,r_I]}}(z_i) &\leq  p^{(1-n)(1-t/r)}\|.\|_{\alpha^r,A'_{[s_I,r_I]}}(r_i)^{t/r}\\
 &\leq p^{-i}p^{(1-n)(1-t/r)}\|.\|_{\alpha^r,A'_{[s_I,r_I]}}(r_i)^{t/r},	
\end{align}
which shows the corresponding sum $\sum_i z_i$ converges to desired element $z$ in the corresponding union of all the Robba rings $\widetilde{\Pi}^{[s,r']}_{R,A'_{[s_I,r_I]}}$ for all $r'\geq r$. Then for the second statement we just perform as before to look at the corresponding lifting of elements through the corresponding strict quotient morphism.	
\end{proof}

\begin{proposition} \mbox{\bf{(After Kedlaya-Liu \cite[Lemma 5.2.10]{KL1})}}
We could have the following result in our new context around the corresponding period rings with respect to two intervals $[s_1,r_1],[s_2,r_2]$ such that we have:
\begin{align}
0<s_1\leq s_2\leq r_1 \leq r_2.	
\end{align}
To be more precise in our current situation, we have:
\begin{align}
\widetilde{\Pi}^{[s_1,r_1]}_{R,A_{[s_I,r_I]}}\bigcap \widetilde{\Pi}^{[s_2,r_2]}_{R,A_{[s_I,r_I]}}=	\widetilde{\Pi}^{\mathrm{int},r}_{R,A_{[s_1,r_2]}}	
\end{align}
and 
\begin{align}
\widetilde{\Pi}^{[s_1,r_1]}_{R,\overline{A_{[s_I,r_I]}}}\bigcap \widetilde{\Pi}^{[s_2,r_2]}_{R,\overline{A_{[s_I,r_I]}}}=	\widetilde{\Pi}^{\mathrm{int},r}_{R,\overline{A_{[s_I,r_I]}}}.	
\end{align}

\end{proposition}

\begin{proof}
See the proof of \cite[Lemma 5.2.10]{KL1}.	
\end{proof}

\begin{proposition} \mbox{\bf{(After Kedlaya-Liu \cite[Lemma 5.2.10]{KL1})}}
We could have the following result in our new context around the corresponding period rings with respect to two intervals $[s_1,r_1],[s_2,r_2]$ such that we have:
\begin{align}
0<s_1\leq s_2\leq r_1 \leq r_2.	
\end{align}
To be more precise in our current situation, we have:
\begin{align}
\widetilde{\Pi}^{[s_1,r_1]}_{R,A'_{[s_I,r_I]}}\bigcap \widetilde{\Pi}^{[s_2,r_2]}_{R,A'_{[s_I,r_I]}}=	\widetilde{\Pi}^{\mathrm{int},r}_{R,A'_{[s_1,r_2]}}	
\end{align}
and 
\begin{align}
\widetilde{\Pi}^{[s_1,r_1]}_{R,\overline{A'_{[s_I,r_I]}}}\bigcap \widetilde{\Pi}^{[s_2,r_2]}_{R,\overline{A'_{[s_I,r_I]}}}=	\widetilde{\Pi}^{\mathrm{int},r}_{R,\overline{A'_{[s_I,r_I]}}}.	
\end{align}

\end{proposition}

\begin{proof}
See the proof of the previous proposition.	
\end{proof}


\newpage

\section{Period Rings and Sheaves with coefficient in Adic Spaces}

\subsection{Fundamental Settings}

\noindent Scholze's diamonds \cite{Sch1} are very general adic stacks where the category is big and convenient enough to include the corresponding perfectoid spaces and the corresponding seminormal rigid analytic spaces. The two spaces happen to be our main interests at the same time. However, we do not actually use the language of diamonds.

\indent Now we use the notation $D$ to denote a general adic space over $E$ in the sense of \cite{KL1} and \cite{KL2}. Then we define the following period rings with coefficient in $D$ by using the corresponding Banach algebra structures over $E$.

\begin{definition}

We now consider the following sheaves in adic topology, \'etale topology and pro-\'etale topology (which are indeed not just presheaves by using the corresponding Schauder basis):
\begin{align}
\widetilde{\Pi}_{R,D},\widetilde{\Pi}^\infty_{R,D},\widetilde{\Pi}^I_{R,D},\widetilde{\Pi}^r_{R,D}	
\end{align}
which assign each affinoid subdomain or perfectoid subdomain $\mathrm{Spa}(A,A^+)$ of $D$ to the following rings:
\begin{align}
\widetilde{\Pi}_{R,A},\widetilde{\Pi}^\infty_{R,A},\widetilde{\Pi}^I_{R,A},\widetilde{\Pi}^r_{R,A}.	
\end{align}

\end{definition}

\indent To define Frobenius modules over $D$ we first define the following local picture:

\begin{definition} \mbox{\bf{(After Kedlaya-Liu \cite[Definition 4.4.4]{KL2})}}
We define the corresponding $\varphi^a$-modules over the corresponding period rings above, which we could use some uniform notation $\triangle^?,?=\emptyset,I,r,\infty,\triangle=\widetilde{\Pi}_{R,B}$ to denote these. We define any $\varphi^a$-module over the period rings with index $\emptyset$ above to be a finitely generated projective module over $\triangle^?$ carrying semilinear action from the Frobenius operator $\varphi^a$ such that $\varphi^{a*}M\overset{\sim}{\rightarrow}M$. We define any $\varphi^a$-module over the period rings with index $r$ above to be a finitely generated projective module over $\triangle^?$ carrying semilinear action from the Frobenius operator $\varphi^a$ such that $\varphi^{a*}M\overset{\sim}{\rightarrow}M\otimes \triangle^{rp^{-ah}}$. We define any $\varphi^a$-module over the period rings with index $[s,r]$ above to be a finitely generated projective module over $\triangle^?$ carrying semilinear action from the Frobenius operator $\varphi^a$ such that $\varphi^{a*}M\otimes_{\triangle^{[sp^{-ha},rp^{-ha}]}} \triangle^{[s,rp^{-ha}]}\overset{\sim}{\rightarrow}M\otimes \triangle^{[s,rp^{-ha}]}$. For a corresponding object over $\triangle^\emptyset$, we assume the module is the base change from some module over $\triangle^{r_0}$ for some $r_0>0$.
\end{definition}

\begin{definition} \mbox{\bf{(After Kedlaya-Liu \cite[Definition 4.4.4]{KL2})}}
We define the corresponding pseudocoherent $\varphi^a$-modules over the corresponding period rings above, which we could use some uniform notation $\triangle^?,?=\emptyset,I,r,\infty,\triangle=\widetilde{\Pi}_{R,B}$ to denote these. We define any pseudocoherent $\varphi^a$-module over the period rings with index $\empty$ above to be a pseudocoherent module over $\triangle^?$ carrying semilinear action from the Frobenius operator $\varphi^a$ such that $\varphi^{a*}M\overset{\sim}{\rightarrow}M$. We define any pseudocoherent $\varphi^a$-module over the period rings with index $r$ above to be a pseudocoherent module over $\triangle^?$ carrying semilinear action from the Frobenius operator $\varphi^a$ such that $\varphi^{a*}M\overset{\sim}{\rightarrow}M\otimes \triangle^{rp^{-ah}}$. We define any pseudocoherent $\varphi^a$-module over the period rings with index $[s,r]$ above to be a pseudocoherent module over $\triangle^?$ carrying semilinear action from the Frobenius operator $\varphi^a$ such that $\varphi^{a*}M\otimes_{\triangle^{[sp^{-ha},rp^{-ha}]}} \triangle^{[s,rp^{-ha}]}\overset{\sim}{\rightarrow}M\otimes \triangle^{[s,rp^{-ha}]}$. For a corresponding object over $\triangle^\emptyset$, we assume the module is the base change from some module over $\triangle^{r_0}$ for some $r_0>0$. As in \cite[Definition 4.4.4]{KL2} we impose the corresponding topological condition on the corresponding modules by imposing that all the modules are complete with respect to the corresponding natural topology, and over the Robba rings with respect to some specific interval we assume that the corresponding modules are \'etale-stably pseudocoherent. And we assume the corresponding stability with respect to the base change as in \cite[Theorem 4.6.1]{KL2} for modules over $\triangle^{r_0},\triangle^\infty$.
\end{definition}


\begin{definition} \mbox{\bf{(After Kedlaya-Liu \cite[Definition 4.4.6]{KL2})}}
We now define the corresponding $\varphi^a$-bundles over the corresponding period rings above, which we could use some uniform notation $\triangle_{*,B}^?,?=\emptyset,r,*=R$ to denote these. We define a $\varphi^a$-bundle over $\triangle_{*,B}^?,?=\emptyset,r$ to be a compatible family of finitely generated projective $\varphi^a$-modules over $\triangle_{*,B}^?,?=[s',r']$ for suitable $[s',r']$ (namely contained in $(0,r]$ if we are looking at the ring $\triangle_{*,B}^r$) satisfying the corresponding restriction compatibility and cocycle condition.
\end{definition}

\begin{definition} \mbox{\bf{(After Kedlaya-Liu \cite[Definition 4.4.6]{KL2})}}
We now define the corresponding pseudocoherent $\varphi^a$-bundles over the corresponding period rings above, which we could use some uniform notation $\triangle_{*,B}^?,?=\emptyset,r,*=R$ to denote these. We define a pseudocoherent $\varphi^a$-bundle over $\triangle_{*,B}^?,?=\emptyset,r$ to be a compatible family of \'etale-stably pseudocoherent $\varphi^a$-modules over $\triangle_{*,B}^?,?=[s',r']$ for suitable $[s',r']$ (namely contained in $(0,r]$ if we are looking at the ring $\triangle_{*,B}^r$) satisfying the corresponding restriction compatibility and cocycle condition.
\end{definition}

\begin{remark}
The $p$-adic functional analytic property we imposed on the algebraic pseudocoherent modules could be also translated to adic topology, which is the corresponding stably pseudocoherent one. This is very important whenever one works with the corresponding analytic topology instead of the corresponding pro-\'etale topology.	
\end{remark}

\indent Then we define things over the adic spaces.

\begin{definition}
In adic topology, we define the $\varphi^a$-module over any $\triangle^?,?=\emptyset,I,r,\infty,\triangle=\widetilde{\Pi}_{R,D}$ to be a compatible family of $\varphi^a$-modules over $\triangle^?,?=\emptyset,I,r,\infty,\triangle=\widetilde{\Pi}_{R,A}$ for each affinoid subdomain $\mathrm{Spa}(A,A^+)$ of $D$.	
\end{definition}

\begin{definition}
In adic topology, we define a pseudocoherent $\varphi^a$-module over any $\triangle^?,?=\emptyset,I,r,\infty,\triangle=\widetilde{\Pi}_{R,D}$ to be a compatible family of pseudocoherent $\varphi^a$-modules over $\triangle^?,?=\emptyset,I,r,\infty,\triangle=\widetilde{\Pi}_{R,A}$ for each affinoid subdomain $\mathrm{Spa}(A,A^+)$ of $D$.	
\end{definition}

\begin{definition}
In adic topology, we define the $\varphi^a$-bundle over any $\triangle^?,?=\emptyset,I,r,\infty,\triangle=\widetilde{\Pi}_{R,D}$ to be a compatible family of $\varphi^a$-bundles over $\triangle^?,?=\emptyset,I,r,\infty,\triangle=\widetilde{\Pi}_{R,A}$ for each affinoid subdomain $\mathrm{Spa}(A,A^+)$ of $D$.	
\end{definition}

\begin{definition}
In adic topology, we define a pseudocoherent $\varphi^a$-bundle over any $\triangle^?,?=\emptyset,I,r,\infty,\triangle=\widetilde{\Pi}_{R,D}$ to be a compatible family of pseudocoherent $\varphi^a$-bundles over $\triangle^?,?=\emptyset,I,r,\infty,\triangle=\widetilde{\Pi}_{R,A}$ for each affinoid subdomain $\mathrm{Spa}(A,A^+)$ of $D$.	
\end{definition}

\begin{definition}
In pro-\'etale topology, we define the $\varphi^a$-module over any $\triangle^?,?=\emptyset,I,r,\infty,\triangle=\widetilde{\Pi}_{R,D}$ to be a compatible family of $\varphi^a$-modules over $\triangle^?,?=\emptyset,I,r,\infty,\triangle=\widetilde{\Pi}_{R,A}$ for each perfectoid subdomain $\mathrm{Spa}(A,A^+)$ of $D$.	
\end{definition}

\begin{definition}
In pro-\'etale topology, we define a pseudocoherent $\varphi^a$-module over any $\triangle^?,?=\emptyset,I,r,\infty,\triangle=\widetilde{\Pi}_{R,D}$ to be a compatible family of pseudocoherent $\varphi^a$-modules over $\triangle^?,?=\emptyset,I,r,\infty,\triangle=\widetilde{\Pi}_{R,A}$ for each perfectoid subdomain $\mathrm{Spa}(A,A^+)$ of $D$.	
\end{definition}

\begin{definition}
In pro-\'etale topology, we define the $\varphi^a$-bundle over any $\triangle^?,?=\emptyset,I,r,\infty,\triangle=\widetilde{\Pi}_{R,D}$ to be a compatible family of $\varphi^a$-bundles over $\triangle^?,?=\emptyset,I,r,\infty,\triangle=\widetilde{\Pi}_{R,A}$ for each perfectoid subdomain $\mathrm{Spa}(A,A^+)$ of $D$.	
\end{definition}

\begin{definition}
In pro-\'etale topology, we define a pseudocoherent $\varphi^a$-bundle over any $\triangle^?,?=\emptyset,I,r,\infty,\triangle=\widetilde{\Pi}_{R,D}$ to be a compatible family of pseudocoherent $\varphi^a$-bundles over $\triangle^?,?=\emptyset,I,r,\infty,\triangle=\widetilde{\Pi}_{R,A}$ for each perfectoid subdomain $\mathrm{Spa}(A,A^+)$ of $D$.	
\end{definition}

\subsection{Fundamental Comparisons} \label{section3.2}

\begin{theorem} \mbox{\bf{(After Kedlaya-Liu \cite[Theorem 4.6.1]{KL2})}} \label{theorem3.10}
Consider the following categories:\\
1. The corresponding category of all the sheaves of pseudocoherent $\mathcal{O}_{\text{\'etale}}$-modules over the adic Fargues-Fontaine curve $Y_{\mathrm{FF},R,D}$ in the corresponding \'etale topology;\\
2. The corresponding category of all the sheaves of pseudocoherent $\mathcal{O}_{\text{pro-\'etale}}$-modules over the adic Fargues-Fontaine curve $Y_{\mathrm{FF},R,D}$ in the corresponding pro-\'etale topology;\\
3. The corresponding category of all the pseudocoherent bundles over the period ring $\widetilde{\Pi}_{R,D}$, carrying the corresponding Frobenius action from the operator $\varphi^a$;\\
4. The corresponding category of all the pseudocoherent modules over the period ring $\widetilde{\Pi}^{\infty}_{R,D}$, carrying the corresponding Frobenius action from the operator $\varphi^a$;\\
5. The corresponding category of all the pseudocoherent modules over the period ring $\widetilde{\Pi}^{}_{R,D}$, carrying the corresponding Frobenius action from the operator $\varphi^a$;\\
6. The corresponding category of all the finitely generated projective modules over the period ring $\widetilde{\Pi}^{[s,r]}_{R,D}$, carrying the corresponding Frobenius action from the operator $\varphi^a$, where $0<s\leq r/p^{ah}$.\\
Then we have that they are equivalent.
\end{theorem}

This could be proved by the following theorem which is also considered in \cite{T2} in the corresponding Hodge-Iwasawa theory:

\begin{theorem} \mbox{\bf{(After Kedlaya-Liu \cite[Theorem 4.6.1]{KL2})}}  \label{theorem3.11}
Consider the following categories ($A$ is a perfectoid algebra corresponding to some local chart of $D$ in the previous theorem):\\
1. The corresponding category of all the sheaves of pseudocoherent $\mathcal{O}_{\text{\'etale}}$-modules over the adic Fargues-Fontaine curve $Y_{\mathrm{FF},R,A}$ in the corresponding \'etale topology;\\
2. The corresponding category of all the sheaves of pseudocoherent $\mathcal{O}_{\text{pro-\'etale}}$-modules over the adic Fargues-Fontaine curve $Y_{\mathrm{FF},R,A}$ in the corresponding pro-\'etale topology;\\
3. The corresponding category of all the pseudocoherent bundles over the period ring $\widetilde{\Pi}_{R,A}$, carrying the corresponding Frobenius action from the operator $\varphi^a$;\\
4. The corresponding category of all the pseudocoherent modules over the period ring $\widetilde{\Pi}^{\infty}_{R,A}$, carrying the corresponding Frobenius action from the operator $\varphi^a$;\\
5. The corresponding category of all the pseudocoherent modules over the period ring $\widetilde{\Pi}^{}_{R,A}$, carrying the corresponding Frobenius action from the operator $\varphi^a$ (note that here we assume that the module descends to some module carrying Frobenius over $\widetilde{\Pi}^{r}_{R,A}$ for some radius $r_0>0$, which further base changes to some module carrying Frobenius over $\widetilde{\Pi}^{[s,r]}_{R,A}$ for some interval $[s,r]$ which is assumed to be \'etale-stably pseudocoherent);\\
6. The corresponding category of all the finitely generated projective modules over the period ring $\widetilde{\Pi}^{[s,r]}_{R,A}$, carrying the corresponding Frobenius action from the operator $\varphi^a$, where $0<s\leq r/p^{ah}$.\\
Then we have that they are equivalent.
\end{theorem}

\begin{proof}
The proof is parallel to \cite[Theorem 4.6.1]{KL2}. From 1,2 to 3, the corresponding results follow from \cite[Theorem 4.4.3]{KL2} due to the natural sheafiness since we are working over perfectoid spaces. Then to glue a bundle to get a module over the Robba ring over the corresponding Robba ring $\widetilde{\Pi}_{R,A}^\infty$, we use the corresponding reified adic space $\mathrm{Spra}(\widetilde{\Pi}_{R,A}^{[sp^{-kah},rp^{-kah}]},\widetilde{\Pi}_{R,A}^{[sp^{-kah},rp^{-kah}],\mathrm{Gr}})$ to cover the corresponding whole spaces, and use the corresponding Frobenius pullback to basically control the corresponding finiteness of the corresponding sections over each member in this covering, then we can derive the result from \cite[Proposition 2.6.17]{KL2}. The same strategy allows us to go from 3 to 4,5. The corresponding equivalence between 3 and 6 could be proved by the following argument. First from 3 to 6 we just take the corresponding projection. Then back we consider the corresponding Frobenius to control finiteness over each $\mathrm{Spra}(\widetilde{\Pi}_{R,A}^{[sp^{-kah},rp^{-kah}]},\widetilde{\Pi}_{R,A}^{[sp^{-kah},rp^{-kah}],\mathrm{Gr}})$ for any $k\geq 0$, then glueing over any arbitrary interval by using the sheafiness of the period rings. One also checks the corresponding pseudocoherence as in \cite[Theorem 4.6.1]{KL2}.
\end{proof}

\begin{corollary} \label{corollary3.12}
We have the following categories are equivalent predicted in \cite{CKZ} partially and implicitly (here $R'$ is a perfect algebra with the same type as that of $R$ but needs not to be the same as $R$):	\\
1. The corresponding category of all the sheaves of pseudocoherent $\mathcal{O}_{\text{\'etale}}$-modules over the adic Fargues-Fontaine curve $Y_{\mathrm{FF},R,Y_{\mathrm{FF},R'}}$ in the corresponding \'etale topology;\\
2. The corresponding category of all the sheaves of pseudocoherent $\mathcal{O}_{\text{pro-\'etale}}$-modules over the adic Fargues-Fontaine curve $Y_{\mathrm{FF},R,Y_{\mathrm{FF},R'}}$ in the corresponding pro-\'etale topology;\\
3. The corresponding category of all the pseudocoherent bundles over the period ring $\widetilde{\Pi}_{R,Y_{\mathrm{FF},R'}}$, carrying the corresponding Frobenius action from the operator $\varphi^a$;\\
4. The corresponding category of all the pseudocoherent modules over the period ring $\widetilde{\Pi}^{\infty}_{R,Y_{\mathrm{FF},R'}}$, carrying the corresponding Frobenius action from the operator $\varphi^a$;\\
5. The corresponding category of all the pseudocoherent modules over the period ring $\widetilde{\Pi}^{}_{R,Y_{\mathrm{FF},R'}}$, carrying the corresponding Frobenius action from the operator $\varphi^a$;\\
6. The corresponding category of all the finitely generated projective modules over the period ring $\widetilde{\Pi}^{[s,r]}_{R,Y_{\mathrm{FF},R'}}$, carrying the corresponding Frobenius action from the operator $\varphi^a$, where $0<s\leq r/p^{ah}$.
\end{corollary}

\indent For some application one could consider the corresponding preperfectoid coefficients:

\begin{theorem} \mbox{\bf{(After Kedlaya-Liu \cite[Theorem 4.6.1]{KL2})}}  
Consider the following categories ($A$ is a preperfectoid algebra over $E$):\\
1. The corresponding category of all the sheaves of pseudocoherent $\mathcal{O}_{\text{\'etale}}$-modules over the adic Fargues-Fontaine curve $Y_{\mathrm{FF},R,A}$ in the corresponding \'etale topology;\\
2. The corresponding category of all the sheaves of pseudocoherent $\mathcal{O}_{\text{pro-\'etale}}$-modules over the adic Fargues-Fontaine curve $Y_{\mathrm{FF},R,A}$ in the corresponding pro-\'etale topology;\\
3. The corresponding category of all the pseudocoherent bundles over the period ring $\widetilde{\Pi}_{R,A}$, carrying the corresponding Frobenius action from the operator $\varphi^a$;\\
4. The corresponding category of all the pseudocoherent modules over the period ring $\widetilde{\Pi}^{\infty}_{R,A}$, carrying the corresponding Frobenius action from the operator $\varphi^a$;\\
5. The corresponding category of all the pseudocoherent modules over the period ring $\widetilde{\Pi}^{}_{R,A}$, carrying the corresponding Frobenius action from the operator $\varphi^a$ (note that here we assume that the module descends to some module carrying Frobenius over $\widetilde{\Pi}^{r}_{R,A}$ for some radius $r>0$, which further base changes to some module carrying Frobenius over $\widetilde{\Pi}^{[s,r]}_{R,A}$ for some interval $[s,r]$ which is assumed to be \'etale-stably pseudocoherent);\\
6. The corresponding category of all the finitely generated projective modules over the period ring $\widetilde{\Pi}^{[s,r]}_{R,A}$, carrying the corresponding Frobenius action from the operator $\varphi^a$, where $0<s\leq r/p^{ah}$.\\
Then we have that they are equivalent.
\end{theorem}

\begin{proof}
The proof is parallel to \cite[Theorem 4.6.1]{KL2}. From 1,2 to 3, the corresponding results follow from \cite[Theorem 4.4.3]{KL2} due to the natural sheafiness since we are working over preperfectoid spaces. Then to glue a bundle to get a module over the Robba ring over the corresponding Robba ring $\widetilde{\Pi}_{R,A}^\infty$, we use the corresponding reified adic space $\mathrm{Spra}(\widetilde{\Pi}_{R,A}^{[sp^{-kah},rp^{-kah}]},\widetilde{\Pi}_{R,A}^{[sp^{-kah},rp^{-kah}],\mathrm{Gr}})$ to cover the corresponding whole spaces, and use the corresponding Frobenius pullback to basically control the corresponding finiteness of the corresponding sections over each member in this covering, then we can derive the result from \cite[Proposition 2.6.17]{KL2}. The same strategy allows us to go from 3 to 4,5. The corresponding equivalence between 3 and 6 could be proved by the following argument. First from 3 to 6 we just take the corresponding projection. Then back we consider the corresponding Frobenius to control finiteness over each $\mathrm{Spra}(\widetilde{\Pi}_{R,A}^{[sp^{-kah},rp^{-kah}]},\widetilde{\Pi}_{R,A}^{[sp^{-kah},rp^{-kah}],\mathrm{Gr}})$ for any $k\geq 0$, then glueing over any arbitrary interval by using the sheafiness of the period rings. One also checks the corresponding pseudocoherence as in \cite[Theorem 4.6.1]{KL2}.\\
\end{proof}

\begin{definition} \mbox{\bf{(After Kedlaya-Liu \cite[Definition 4.4.4]{KL2})}}
We define the corresponding $\varphi^a$-$\varphi^{'a}$-modules over the corresponding period rings, which we could use some uniform notation $\triangle^?\widehat{\otimes}\triangle^{?'},?,?'=\emptyset,I,r,\infty,\triangle=\widetilde{\Pi}_{R}$ to denote these. We define any $\varphi^a$-$\varphi^{'a}$-module over the period rings with index $\emptyset$ above to be a finitely generated projective module over $\triangle^?\widehat{\otimes}\triangle^{?'}$ carrying semilinear action from the Frobenius operator $\varphi^a$ such that $\varphi^{a*}M\overset{\sim}{\rightarrow}M$ where $\varphi^a$ comes from the first factor, and carrying semilinear action from the Frobenius operator $\varphi^{'a}$ such that $\varphi^{'a*}M\overset{\sim}{\rightarrow}M$ where $\varphi^{'a}$ comes from the second factor, and we assume that this module descends to some $\triangle^r\widehat{\otimes}\triangle^{r'}$ for $0<r,r'<\infty$. We define any $\varphi^a$-$\varphi^{'a}$-module over the period rings with index $r_1,r_2$ above to be a finitely generated projective module over $\triangle^?\widehat{\otimes}\triangle^{?'}$ carrying semilinear action from the Frobenius operator $\varphi^a$ such that $\varphi^{a*}M\overset{\sim}{\rightarrow}M\otimes \triangle^{r_1p^{-ah}}\widehat{\otimes}\triangle^{?'}$, carrying semilinear action from the Frobenius operator $\varphi^{'a}$ such that $\varphi^{'a*}M\overset{\sim}{\rightarrow}M\otimes \triangle^?\widehat{\otimes}\triangle^{r_2p^{-ah}}$. We define any $\varphi^a$-$\varphi^{'a}$-module over the period rings with index $[s_1,r_1],[s_2,r_2]$ above to be a finitely generated projective module over $\triangle^?\widehat{\otimes}\triangle^{?'}$ carrying semilinear action from the Frobenius operator $\varphi^a$ such that $\varphi^{a*}M\otimes_{\triangle^{[s_1p^{-ha},r_1p^{-ha}]}\widehat{\otimes}\triangle^{?'}} \triangle^{[s_1,r_1p^{-ha}]} \widehat{\otimes}\triangle^{?'}\overset{\sim}{\rightarrow}M\otimes \triangle^{[s_1,r_1p^{-ha}]}\widehat{\otimes}\triangle^{?'}$, carrying semilinear action from the Frobenius operator $\varphi^{'a}$ such that $\varphi^{'a*}M\otimes_{\triangle^?\widehat{\otimes}\triangle^{[s_2p^{-ha},r_2p^{-ha}]}} \triangle^?\widehat{\otimes}\triangle^{[s_2,r_2p^{-ha}]}\overset{\sim}{\rightarrow}M\otimes \triangle^?\widehat{\otimes}\triangle^{[s_2,r_2p^{-ha}]}$. For a corresponding object over $\triangle^\emptyset\widehat{\otimes}\triangle^\emptyset$, we assume the module is the base change from some module over $\triangle^{r_0}\widehat{\otimes}\triangle^{r'_0}$ for some $r_0,r_0'>0$, as mentioned above. 
\end{definition}

\begin{definition} \mbox{\bf{(After Kedlaya-Liu \cite[Definition 4.4.4]{KL2})}}
We define the corresponding pseudocoherent $\varphi^a$-$\varphi^{'a}$-modules over the corresponding period rings above, which we could use some uniform notation $\triangle^?\widehat{\otimes}\triangle^{?'},?,?'=\emptyset,I,r,\infty,\triangle=\widetilde{\Pi}_{R}$ to denote these. We define any pseudocoherent $\varphi^a$-$\varphi^{'a}$-module over the period rings with index $\emptyset$ above to be a pseudocoherent module over $\triangle^?\widehat{\otimes}\triangle^{?'}$ carrying semilinear action from the Frobenius operator $\varphi^a$ such that $\varphi^{a*}M\overset{\sim}{\rightarrow}M$ where $\varphi^a$ comes from the first factor, and carrying semilinear action from the Frobenius operator $\varphi^{'a}$ such that $\varphi^{'a*}M\overset{\sim}{\rightarrow}M$ where $\varphi^{'a}$ comes from the second factor, and we assume that this module descends to some $\triangle^r\widehat{\otimes}\triangle^{r'}$ for $0<r,r'<\infty$. We define any pseudocoherent $\varphi^a$-$\varphi^{'a}$-module over the period rings with index $r_1,r_2$ above to be a pseudocoherent module over $\triangle^?\widehat{\otimes}\triangle^{?'}$ carrying semilinear action from the Frobenius operator $\varphi^a$ such that $\varphi^{a*}M\overset{\sim}{\rightarrow}M\otimes \triangle^{r_1p^{-ah}}\widehat{\otimes}\triangle^{?'}$, carrying semilinear action from the Frobenius operator $\varphi^{'a}$ such that $\varphi^{'a*}M\overset{\sim}{\rightarrow}M\otimes \triangle^?\widehat{\otimes}\triangle^{r_2p^{-ah}}$. We define any pseudocoherent $\varphi^a$-$\varphi^{'a}$-module over the period rings with index $[s_1,r_1],[s_2,r_2]$ above to be a pseudocoherent module over $\triangle^?\widehat{\otimes}\triangle^{?'}$ carrying semilinear action from the Frobenius operator $\varphi^a$ such that $\varphi^{a*}M\otimes_{\triangle^{[s_1p^{-ha},r_1p^{-ha}]}\widehat{\otimes}\triangle^{?'}} \triangle^{[s_1,r_1p^{-ha}]} \widehat{\otimes}\triangle^{?'}\overset{\sim}{\rightarrow}M\otimes \triangle^{[s_1,r_1p^{-ha}]}\widehat{\otimes}\triangle^{?'}$, carrying semilinear action from the Frobenius operator $\varphi^{'a}$ such that $\varphi^{'a*}M\otimes_{\triangle^?\widehat{\otimes}\triangle^{[s_2p^{-ha},r_2p^{-ha}]}} \triangle^?\widehat{\otimes}\triangle^{[s_2,r_2p^{-ha}]}\overset{\sim}{\rightarrow}M\otimes \triangle^?\widehat{\otimes}\triangle^{[s_2,r_2p^{-ha}]}$. For a corresponding object over $\triangle^\emptyset\widehat{\otimes}\triangle^\emptyset$, we assume the module is the base change from some module over $\triangle^{r_0}\widehat{\otimes}\triangle^{r'_0}$ for some $r_0,r_0'>0$. As in \cite[Definition 4.4.4]{KL2} we impose the corresponding topological condition on the corresponding modules by imposing that all the modules are complete with respect to the corresponding natural topology, and over the Robba rings with respect to some specific multi-interval we assume that the corresponding modules are \'etale-stably pseudocoherent.
\end{definition}

\begin{theorem}
We have the following categories are equivalent predicted in \cite{CKZ}:	\\
1. The corresponding category of all the sheaves of pseudocoherent $\mathcal{O}_{\text{\'etale}}$-modules over the adic Fargues-Fontaine curve $Y_{\mathrm{FF},R,Y_{\mathrm{FF},R}}$ in the corresponding \'etale topology;\\
2. The corresponding category of all the sheaves of pseudocoherent $\mathcal{O}_{\text{pro-\'etale}}$-modules over the adic Fargues-Fontaine curve $Y_{\mathrm{FF},R,Y_{\mathrm{FF},R}}$ in the corresponding pro-\'etale topology;\\
3. The corresponding category of all the pseudocoherent modules over the period ring $\widetilde{\Pi}^{[s,r]}_{R}\widehat{\otimes}\widetilde{\Pi}^{[s',r']}_{R}$, carrying the corresponding partial Frobenius action from the operators $\varphi^a$ and $\varphi^{'a}$ from the first $R$ and the second $R$ respectively, here we assume that $0<s\leq r/p^{ah}<\infty$ and $0<s'\leq r'/p^{ah}<\infty$ (which are certainly assumed to be \'etale-stably pseudocoherent).\\
\end{theorem}

\begin{proof}
From 1,2 to 3, we consider the corresponding projection to the subspace with respect to the first factor namely:
\begin{align}
\mathrm{Spa}(\widetilde{\Pi}^{[s,r]}_{R,Y_{\mathrm{FF},R}},\widetilde{\Pi}^{[s,r],+}_{R,Y_{\mathrm{FF},R}})\slash \varphi^\mathbb{Z}.	
\end{align}
By the quasi-compactness of $Y_{\mathrm{FF},R}$ we have that for some fixed nice $[s,r]$ the pseudocoherent $\mathcal{O}$-modules could be regarded in some equivalent way as pseudocoherent $\widetilde{\Pi}^{[s,r]}_{R,Y_{\mathrm{FF},R}}$-modules carrying the partial Frobenius action from $\varphi^a$ (namely the corresponding finiteness is preserved during this functor). Then consider the preperfectoid coefficients in $\widetilde{\Pi}^{[s,r]}_{R}$ we have that this category could be further related to the one of all the pseudocoherent $\widetilde{\Pi}^{[s,r]}_{R}\widehat{\otimes}\widetilde{\Pi}^{[s',r']}_{R}$-modules carrying the partial Frobenius action from $\varphi^a$ and the partial Frobenius action from $\varphi^{'a}$. We then start from any object in 3, we then use the corresponding reified adic space:
\begin{align}
\mathrm{Spra}(\widetilde{\Pi}^{[sp^{-kha},rp^{-kha}]}_{R}\widehat{\otimes}\widetilde{\Pi}^{[s'p^{-k'ha},r'p^{-k'ha}]}_{R},(\widetilde{\Pi}^{[sp^{-kha},rp^{-kha}]}_{R}\widehat{\otimes}\widetilde{\Pi}^{[s'p^{-k'ha},r'p^{-k'ha}]}_{R})^\mathrm{Gr})	
\end{align}
to cover the corresponding spaces showing the corresponding essential surjectivity of the functor from 1 to 3. Again we need to apply \cite[Proposition 2.6.17]{KL2} to obtain the finiteness around the global section.\\
\end{proof}


\newpage

\section{Sheafiness and $\infty$-Period Rings and Sheaves} \label{section4}

\subsection{The $\infty$-Period Rings and Sheaves}

\noindent Here we would like to mention something related to the corresponding sheafiness of the Robba rings in our context. Namely when we discuss the corresponding modules over period rings with big coefficients we might come across the issues around the sheafiness. However recent work \cite{BK} could let us get around the issue by just considering the $\infty$-adic space $(\mathrm{Spa}^h(S), \mathcal{O}^\mathrm{der}_{\mathrm{Spa}^h(S)})$ attached a Banach algebra $S$, which is just an $\infty$-analytic stack.

\indent Following \cite[Introduction]{BK} we define the following $\infty$-analytic space which are the corresponding coverings of the $\infty$-version of Fargues-Fontaine curves as those considered in \cite{KL1}, \cite{KL2} and \cite{T2}. Here we drop the corresponding sheafiness condition on the ring $\widetilde{\Pi}^{[s,r]}_{R,A}$ for any closed subinterval of $(0,\infty)$. Then we define:
\begin{displaymath}
X^\infty_{\mathrm{FF},R,A}:= \varinjlim_{0<s\leq r< \infty} \mathrm{Spa}^h(\widetilde{\Pi}^{[s,r]}_{R,A})	
\end{displaymath}
which carries a corresponding derived sheaves of $\infty$-rings, namely we have ringed $\infty$-adic space:
\begin{align}
(\varinjlim_{0<s\leq r< \infty} \mathrm{Spa}^h(\widetilde{\Pi}^{[s,r]}_{R,A}), \varprojlim_{0<s\leq r< \infty}\mathcal{O}^\mathrm{der}_{\mathrm{Spa}^h(\widetilde{\Pi}^{[s,r]}_{R,A})}).	
\end{align}

\begin{remark}
In \cite{BK}, there are two different notions of the corresponding $\infty$-Huber spectra attached to Banach algebras. We choose to focus on $(\mathrm{Spa}^h_\mathrm{Rat}(S), \mathcal{O}^\mathrm{der}_{\mathrm{Spa}_\mathrm{Rat}^h(S)})$ as our $(\mathrm{Spa}^h(S), \mathcal{O}^\mathrm{der}_{\mathrm{Spa}^h(S)})$ in this section. Recall the corresponding Tate acyclicity in this case still holds for derived standard rational localization.	
\end{remark}

\begin{setting}
Note (!) that in our current section we are going to use the notation $A$ to denote any commutative Banach uniform adic algebra over $\mathbb{Q}_p$.
\end{setting}

\begin{definition}
We define the corresponding $\infty$ deformed Fargues-Fontaine curves in general notation $Y^\infty_{\mathrm{FF},R,A}$ to be the corresponding quotient of the space by the $\mathbb{Z}$-power of the Frobenius operator which acts through the corresponding affinoids on the algebraic level.	
\end{definition}

\indent The presheaf of $\infty$-rings $\varprojlim_{0<s\leq r< \infty}\mathcal{O}^\mathrm{der}_{\mathrm{Spa}^h(\widetilde{\Pi}^{[s,r]}_{R,A})}$ are general limit of the corresponding sheaves $\mathcal{O}^\mathrm{der}_{\mathrm{Spa}^h(\widetilde{\Pi}^{[s,r]}_{R,A})}$, which provides upgrading of the discussion in \cite{T2}. So now we consider similarly as above the following $\infty$-adic spaces:

\begin{align}
\mathrm{Spa}^h\widetilde{\Pi}_{R,A},\mathrm{Spa}^h\widetilde{\Pi}^\infty_{R,A},\mathrm{Spa}^h\widetilde{\Pi}^I_{R,A},\mathrm{Spa}^h\widetilde{\Pi}^r_{R,A}.	
\end{align}

\indent Then take the corresponding global section we have the following $\infty$-period rings:

\begin{align}
\widetilde{\Pi}^h_{R,A},\widetilde{\Pi}^{\infty,h}_{R,A},\widetilde{\Pi}^{I,h}_{R,A},\widetilde{\Pi}^{r,h}_{R,A}.	
\end{align}

\begin{definition}
We define the corresponding Frobenius (on the $\infty$-rings) as in the corresponding non-derived situation through the corresponding affinoids on the algebraic level.	
\end{definition}



\begin{definition} \mbox{\bf{(After Lurie \cite[Corollary 7.2.2.9,Definition 7.2.2.10]{Lu1})}}
For general projective with finiteness condition we look at the following definition following \cite[Corollary 7.2.2.9]{Lu1}. We will use the notion a $f$-$projective$ module over a $\infty$-algebra $S$ to mean a projective object $M$ in the corresponding $\infty$-category of all the $S$-modules such that we have $\pi_0M$ is finite projective over $\pi_0 S$. 	
\end{definition}



\begin{definition} \mbox{\bf{(After Kedlaya-Liu \cite[Definition 4.4.4]{KL2})}}
We define the corresponding $\varphi^a$-modules over the corresponding $\infty$-period rings above, which we could use some uniform notation $\triangle^{?,h},?=\emptyset,I,r,\infty$ to denote these. We define any $\varphi^a$-module over the $\infty$-period rings with index $\emptyset$ above to be a $f$-projective module over $\triangle^{?,h}$ carrying semilinear action from the Frobenius operator $\varphi^a$ such that $\varphi^{a*}M\overset{\sim}{\rightarrow}M$. We define any $\varphi^a$-module over the derived period rings with index $r$ above to be a $f$-projective module over $\triangle^{?,h}$ carrying semilinear action from the Frobenius operator $\varphi^a$ such that $\varphi^{a*}M\overset{\sim}{\rightarrow}M\otimes \triangle^{rp^{-ah},h}$. We define any $\varphi^a$-module over the $\infty$-period rings with index $[s,r]$ above to be a $f$-projective module over $\triangle^{?,h}$ carrying semilinear action from the Frobenius operator $\varphi^a$ such that $\varphi^{a*}M\otimes \triangle^{[s,rp^{-ha}],h}\overset{\sim}{\rightarrow}M\otimes_{\triangle^{[sp^{-ha},rp^{-ha}],h}} \triangle^{[s,rp^{-ha}],h}$. Here $\triangle=\widetilde{\Pi}_{R,A}$. For a corresponding object over $\triangle^{\emptyset,h}$, we assume the module is the base changes from some module over $\triangle^{r_0,h}$ for some $r_0>0$.
\end{definition}

\begin{remark}
We believe that one can discuss more general objects such as the corresponding coherent sheaves after \cite{Lu1} and \cite{Lu2}, but at this moment let us just focus on the derived vector bundles.	
\end{remark}

\begin{definition} \mbox{\bf{(After Kedlaya-Liu \cite[Definition 4.4.6]{KL2})}}
We now define the corresponding $\varphi^a$-bundles over the corresponding $\infty$-period rings above, which we could use some uniform notation $\triangle_{*,A}^{?,h},?=\emptyset,I,r,\infty,*=R$ to denote these. We define a $\varphi^a$-bundle over $\triangle_{*,A}^{?,h},?=\emptyset,r$ to be a compatible family of $f$-projective $\varphi^a$-modules over $\triangle_{*,A}^{?,h},?=[s',r']$ for suitable $[s',r']$ (namely contained in $(0,r]$) satisfying the corresponding restriction compatibility and cocycle condition.
\end{definition}

\subsection{Some Results on the $\infty$-Descent}

\begin{conjecture} \mbox{\bf{(After Kedlaya-Liu \cite[Theorem 4.6.1]{KL2})}}
Consider the following categories:\\
1. The corresponding category of all the bundles over the homotopical period ring $\widetilde{\Pi}^h_{R,A}$, carrying the corresponding Frobenius action from the operator $\varphi^a$;\\
2. The corresponding category of all the $f$-projective modules over the homotopical period ring $\widetilde{\Pi}^{\infty,h}_{R,A}$, carrying the corresponding Frobenius action from the operator $\varphi^a$;\\
3. The corresponding category of all the $f$-projective modules over the homotopical period ring $\widetilde{\Pi}^{[s,r],h}_{R,A}$, carrying the corresponding Frobenius action from the operator $\varphi^a$, where $0<s\leq r/p^{ah}$.\\
Then we have that they are equivalent.
\end{conjecture}

\indent In the corresponding sheafy case, we do have some good comparison:

\begin{theorem} \mbox{\bf{(After Kedlaya-Liu \cite[Theorem 4.6.1]{KL2})}}
In the situation where the Robba ring $\widetilde{\Pi}^{[s,r]}_{R,A}$ for any closed interval $[s,r]$ is stably uniform. Consider the following categories:\\
1. The corresponding category of all the bundles over the homotopical period ring $\widetilde{\Pi}^h_{R,A}$, carrying the corresponding Frobenius action from the operator $\varphi^a$;\\
2. The corresponding category of all the $f$-projective modules over the homotopical period ring $\widetilde{\Pi}^{\infty,h}_{R,A}$, carrying the corresponding Frobenius action from the operator $\varphi^a$;\\
3. The corresponding category of all the $f$-projective modules over the homotopical period ring $\widetilde{\Pi}^{[s,r],h}_{R,A}$, carrying the corresponding Frobenius action from the operator $\varphi^a$, where $0<s\leq r/p^{ah}$.
Then we have that they are equivalent.
\end{theorem}

\begin{proof}
This amounts to the corresponding statement for modules over classical rings. The proof is parallel to \cite[Theorem 4.6.1]{KL2}. Then to glue a bundle to get a module over the Robba ring over the corresponding Robba ring $\widetilde{\Pi}_{R,A}^\infty$, we use the corresponding reified adic space $\mathrm{Spra}(\widetilde{\Pi}_{R,A}^{[sp^{-kah},rp^{-kah}]},\widetilde{\Pi}_{R,A}^{[sp^{-kah},rp^{-kah}],\mathrm{Gr}})$ to cover the corresponding whole spaces, and use the corresponding Frobenius pullback to basically control the corresponding finiteness of the corresponding sections over each member in this covering, then we can derive the result from \cite[Proposition 2.6.17, Corollary 2.6.10]{KL2}. This will show the equivalence between 1 and 2. Then the functor from 1 to 3 is just the corresponding projection. On the other hand the functor from 3 to 1, we use the corresponding Frobenius to reach any interval taking the form of $[sp^{-nha},rp^{-nha}]$ for any $n\in \mathbb{Z}$, then we still have to consider the corresponding extraction of a single module from two over some overlapped such specific intervals coming from the Frobenius. However this is achievable since for $\pi_0$ of the corresponding two modules we are done through \cite[Theorem 1.3.9]{KL1}. 	
\end{proof}

\begin{remark}
The stably-uniform condition here relates directly to the construction of \cite{BK}. 	
\end{remark}

\begin{theorem} \mbox{\bf{(After Kedlaya-Liu \cite[Theorem 4.6.1]{KL2})}} \label{theorem4.12}
Consider the following categories:\\
1. The corresponding category of all the bundles over the homotopical period ring $\widetilde{\Pi}^h_{R,A}$, carrying the corresponding Frobenius action from the operator $\varphi^a$, such that the corresponding underlying modules are plat (here we assume that the modules are $f$-projective);\\
2. The corresponding category of all the $f$-projective modules over the homotopical period ring $\widetilde{\Pi}^{[s,r],h}_{R,A}$, carrying the corresponding Frobenius action from the operator $\varphi^a$, where $0<s\leq r/p^{ah}$, such that the corresponding underlying modules are plat.\\
Then we have that they are equivalent.
\end{theorem}

\begin{proof}
Due to the fact that the $\infty$-presheaf $\mathcal{O}^\mathrm{der}_{\mathrm{Spa}^h(\widetilde{\Pi}^{I,h}_{R,A})}$ is a $\infty$-sheaf, the proof reduces to \cite[Theorem 4.6.1]{KL2} which we also performed in \cite[Section 4.2]{T2}. To compare, the functor from 1 to 2 is just the corresponding projection. On the other hand the functor from 2 to 1, we use the corresponding Frobenius to reach any interval taking the form of $[sp^{-nha},rp^{-nha}]$ for any $n\in \mathbb{Z}$, then we still have to consider the corresponding extraction of a single module from two over some overlapped such specific intervals coming from the Frobenius. However this is achievable since for $\pi_0$ of the corresponding two modules we are done through \cite[Theorem 1.3.9]{KL1}.
\end{proof}

\begin{remark}
Professor Kedlaya has informed us that the condition on the spectrum in \cite[Theorem 1.3.9 (b)]{KL1} could be removed (see our amplification in \cref{section5} below namely \cref{proposition5.11con}). This made us believe above holds. We want to mention here that since as mentioned in \cite{BK} the authors of \cite{BK} believe there is deep relationship between \cite{BK} and Clausen-Scholze's work \cite{CS}. Therefore we believe there is deep relationship between ours and the possible ones after Clausen-Scholze \cite{CS}. 
\end{remark}

\newpage

\section{Noncommutative Descent Revisit} \label{section5}

\subsection{Big Noncommutative Coefficients}

\noindent We now also consider the corresponding more general noncommutative Banach coefficients as in our previous work \cite{T2}, which is targeted at the corresponding $K$-theory of the general Robba rings.

\begin{setting}
Note (!) that in our current section we are going to use the notation $B$ to denote any Banach algebra over $E$.
\end{setting}

\begin{definition} \mbox{\bf{(After Kedlaya-Liu \cite[Definition 4.1.1]{KL2})}}
We let $B$ be a Banach algebra over $E$ with integral subring $\mathcal{O}_B$ over $\mathcal{O}_E$. Recall from the corresponding context in \cite{KL1} we have the corresponding period rings in the relative setting. We follow the corresponding notations we used in \cite[Section 2.1]{T2} for those corresponding period rings. We first have for a pair $(R,R^+)$ where $R$ is a uniform perfect adic Banach ring over the assumed base $\mathcal{O}_F$. Then we take the corresponding generalized Witt vectors taking the form of $W_{\mathcal{O}_E}(R)$, which is just the ring $\widetilde{\Omega}_R^\mathrm{int}$, and by inverting the corresponding uniformizer we have the ring $\widetilde{\Omega}_R$, then by taking the completed product with $B$ we have $\widetilde{\Omega}_{R,B}$. Now for some $r>0$ we consider the ring $\widetilde{\Pi}^{\mathrm{int},r}_{R}$ which is the completion of $W_{\mathcal{O}_E}(R^+)[[r]:r\in R] $ by the norm $\|.\|_{\alpha^r}$ defined by:
\begin{align}
\|.\|_{\alpha^r}(\sum_{n\geq 0}\pi^n[\overline{r}_n])=\sup_{n\geq 0}\{p^{-n}\alpha(\overline{r}_n)^r\}.	
\end{align}
Then we have the product $\widetilde{\Pi}^{\mathrm{bd},r}_{R,B}$ defined as the completion under the product norm $\|.\|_{\alpha^r}\otimes \|.\|_B$ of the corresponding ring $\widetilde{\Pi}^{\mathrm{bd},r}_{R}\otimes_{\mathbb{Q}_p}B$, which could be also defined from the corresponding integral Robba rings defined above. Then we define the corresponding Robba ring for some interval $I\subset (0,\infty)$ with coefficient in the perfectoid ring $B$ denoted by $\widetilde{\Pi}^{I}_{R,B}$ as the following complete tensor product:
\begin{displaymath}
\widetilde{\Pi}^{I}_{R}\widehat{\otimes}_{\mathbb{Q}_p} B	
\end{displaymath}
under the the corresponding tensor product norm $\|.\|_{\alpha^r}\otimes \|.\|_B$. Then we set $\widetilde{\Pi}^r_{R,B}$ as $\varprojlim_{s\rightarrow 0}\widetilde{\Pi}^{[s,r]}_{R,B}$, and then we define $\widetilde{\Pi}_{R,B}$ as $\varinjlim_{r\rightarrow \infty}\widetilde{\Pi}^{[s,r]}_{R,B}$, and we define $\widetilde{\Pi}^\infty_{R,B}$ as $\varprojlim_{r\rightarrow \infty}\widetilde{\Pi}^{r}_{R,B}$. And we also have the corresponding full integral Robba ring and the corresponding full bounded Robba ring by taking the corresponding union through all $r>0$.
\end{definition}

\indent We first define the corresponding right Frobenius modules and bundles:

\begin{definition} \mbox{\bf{(After Kedlaya-Liu \cite[Definition 4.4.4]{KL2})}}
We define the corresponding right $\varphi^a$-modules over the corresponding period rings above, which we could use some uniform notation $\triangle^?,?=\emptyset,I,r,\infty,\triangle=\widetilde{\Pi}_{R,B}$ to denote these. We define any right $\varphi^a$-module over the period rings with index $\empty$ above to be a finitely generated projective right module over $\triangle^?$ carrying semilinear action from the Frobenius operator $\varphi^a$ such that $\varphi^{a*}M\overset{\sim}{\rightarrow}M$. We define any right $\varphi^a$-module over the period rings with index $r$ above to be a finitely generated projective right module over $\triangle^?$ carrying semilinear action from the Frobenius operator $\varphi^a$ such that $\varphi^{a*}M\overset{\sim}{\rightarrow}M\otimes \triangle^{rp^{-ah}}$. We define any right $\varphi^a$-module over the period rings with index $[s,r]$ above to be a finitely generated projective right module over $\triangle^?$ carrying semilinear action from the Frobenius operator $\varphi^a$ such that $\varphi^{a*}M\otimes_{\triangle^{[sp^{-ha},rp^{-ha}]}} \triangle^{[s,rp^{-ha}]}\overset{\sim}{\rightarrow}M\otimes \triangle^{[s,rp^{-ha}]}$. For a corresponding object over $\triangle^\emptyset$, we assume the module is the base changes from some module over $\triangle^{r_0}$ for some $r_0$.
\end{definition}

\begin{definition} \mbox{\bf{(After Kedlaya-Liu \cite[Definition 4.4.6]{KL2})}}
We now define the corresponding right $\varphi^a$-bundles over the corresponding period rings above, which we could use some uniform notation $\triangle_{*,B}^?,?=\emptyset,I,r,\infty,*=R$ to denote these. We define a right $\varphi^a$-bundle over $\triangle_{*,B}^?,?=\emptyset,r$ to be a compatible family of finitely generated projective right $\varphi^a$-modules over $\triangle_{*,B}^?,?=[s',r']$ for suitable $[s',r']$ (namely contained in $(0,r]$) satisfying the corresponding restriction compatibility and cocycle condition.
\end{definition}

\indent We first define the corresponding left Frobenius modules and bundles:

\begin{definition} \mbox{\bf{(After Kedlaya-Liu \cite[Definition 4.4.4]{KL2})}}
We define the corresponding left $\varphi^a$-modules over the corresponding period rings above, which we could use some uniform notation $\triangle^?,?=\emptyset,I,r,\infty,\triangle=\widetilde{\Pi}_{R,B}$ to denote these. We define any left $\varphi^a$-module over the period rings with index $\empty$ above to be a finitely generated projective left module over $\triangle^?$ carrying semilinear action from the Frobenius operator $\varphi^a$ such that $\varphi^{a*}M\overset{\sim}{\rightarrow}M$. We define any left $\varphi^a$-module over the period rings with index $r$ above to be a finitely generated projective left module over $\triangle^?$ carrying semilinear action from the Frobenius operator $\varphi^a$ such that $\varphi^{a*}M\overset{\sim}{\rightarrow} \triangle^{rp^{-ah}}\otimes M$. We define any left $\varphi^a$-module over the period rings with index $[s,r]$ above to be a finitely generated projective left module over $\triangle^?$ carrying semilinear action from the Frobenius operator $\varphi^a$ such that $\triangle^{[s,rp^{-ha}]} \otimes_{\triangle^{[sp^{-ha},rp^{-ha}]}} \varphi^{a*}M \overset{\sim}{\rightarrow} \triangle^{[s,rp^{-ha}]} \otimes M$. For a corresponding object over $\triangle^{\emptyset}$, we assume the module is the base changes from some module over $\triangle^{r_0}$ for some $r_0>0$.

\end{definition}

\begin{definition} \mbox{\bf{(After Kedlaya-Liu \cite[Definition 4.4.6]{KL2})}}
We now define the corresponding left $\varphi^a$-bundles over the corresponding period rings above, which we could use some uniform notation $\triangle_{*,B}^?,?=\emptyset,I,r,\infty,*=R$ to denote these. We define a left $\varphi^a$-bundle over $\triangle_{*,B}^?,?=\emptyset,r$ to be a compatible family of finitely generated projective left $\varphi^a$-modules over $\triangle_{*,B}^?,?=[s',r']$ for suitable $[s',r']$ (namely contained in $(0,r]$) satisfying the corresponding restriction compatibility and cocycle condition.
\end{definition}

\indent We first define the corresponding Frobenius bimodules and bibundles:

\begin{definition} \mbox{\bf{(After Kedlaya-Liu \cite[Definition 4.4.4]{KL2})}}
We define the corresponding $\varphi^a$-bimodules over the corresponding period rings above, which we could use some uniform notation $\triangle^?,?=\emptyset,I,r,\infty,\triangle=\widetilde{\Pi}_{R,B}$ to denote these. We define any $\varphi^a$-bimodule over the period rings with index $\empty$ above to be a finitely generated projective bimodule over $\triangle^?$ carrying semilinear action from the Frobenius operator $\varphi^a$ such that $\varphi^{a*}M\overset{\sim}{\rightarrow}M$. We define any $\varphi^a$-bimodule over the period rings with index $r$ above to be a finitely generated projective bimodule over $\triangle^?$ carrying semilinear action from the Frobenius operator $\varphi^a$ such that $\varphi^{a*}M\overset{\sim}{\rightarrow} \triangle^{rp^{-ah}}\otimes M$ and $\varphi^{a*}M\overset{\sim}{\rightarrow}M\otimes \triangle^{rp^{-ah}}$. We define any $\varphi^a$-bimodule over the period rings with index $[s,r]$ above to be a finitely generated projective bimodule over $\triangle^?$ carrying semilinear action from the Frobenius operator $\varphi^a$ such that $\triangle^{[s,rp^{-ha}]} \otimes_{\triangle^{[sp^{-ha},rp^{-ha}]}} \varphi^{a*}M \overset{\sim}{\rightarrow} \triangle^{[s,rp^{-ha}]} \otimes M$ and $\varphi^{a*}M\otimes_{\triangle^{[sp^{-ha},rp^{-ha}]}} \triangle^{[s,rp^{-ha}]}\overset{\sim}{\rightarrow}M\otimes \triangle^{[s,rp^{-ha}]}$. For a corresponding object over $\triangle^{\emptyset}$, we assume the module is the base changes from some module over $\triangle^{r_0}$ for some $r_0>0$.

\end{definition}

\begin{definition} \mbox{\bf{(After Kedlaya-Liu \cite[Definition 4.4.6]{KL2})}}
We now define the corresponding $\varphi^a$-bibundles over the corresponding period rings above, which we could use some uniform notation $\triangle_{*,B}^?,?=\emptyset,I,r,\infty,*=R$ to denote these. We define a $\varphi^a$-bibundle over $\triangle_{*,B}^?,?=\emptyset,r$ to be a compatible family of finitely generated projective $\varphi^a$-bimodules over $\triangle_{*,B}^?,?=[s',r']$ for suitable $[s',r']$ (namely contained in $(0,r]$) satisfying the corresponding restriction compatibility and cocycle condition.\\
\end{definition}

\subsection{Glueing Noncommutative Vector Bundles}

\indent The following is what we achieved in the corresponding paper \cite[Lemma 6.82]{T2}:

\begin{proposition}
Consider the following exact sequence of Banach algebras satisfying the corresponding conditions in \cite[Definition 2.7.3 (a),(b)]{KL1}:
\begin{align}
0\rightarrow {\Pi}\rightarrow \Pi_1\bigoplus \Pi_2\rightarrow \Pi_{12}\rightarrow 0.	
\end{align}
And consider the corresponding right glueing datum $(M_1,M_2,M_{12})$ over the corresponding rings above, and we assume that the glueing datum is right finite projective. Then we have that the corresponding kernel $M$ of
\begin{align}
M_1\bigoplus M_2\rightarrow M_{12}	
\end{align}
as right module over $\Pi$ is finitely presented, with the corresponding isomorphisms:
\begin{align}
M\otimes \Pi_1 \overset{\sim}{\rightarrow}	M_1,\\
M\otimes \Pi_2 \overset{\sim}{\rightarrow}	M_2.
\end{align}

\end{proposition}

\begin{proof}
This is essentially proved in \cite[Lemma 6.82]{T2}.	
\end{proof}

\indent By considering the parallel argument one has the following:

\begin{proposition}
Consider the following exact sequence of Banach algebras satisfying the corresponding conditions in \cite[Definition 2.7.3 (a),(b)]{KL1}:
\begin{align}
0\rightarrow {\Pi}\rightarrow \Pi_1\bigoplus \Pi_2\rightarrow \Pi_{12}\rightarrow 0.	
\end{align}
And consider the corresponding left glueing datum $(M_1,M_2,M_{12})$ over the corresponding rings above, and we assume that the glueing datum is left finite projective. Then we have that the corresponding kernel $M$ of
\begin{align}
M_1\bigoplus M_2\rightarrow M_{12}	
\end{align}
as left module over $\Pi$ is finitely presented, with the corresponding isomorphisms:
\begin{align}
M\otimes \Pi_1 \overset{\sim}{\rightarrow}	M_1,\\
M\otimes \Pi_2 \overset{\sim}{\rightarrow}	M_2.
\end{align}

\end{proposition}

\indent Then we consider the corresponding bimodules, and by relying on some argument essentially due to Kedlaya we have the following:

\begin{proposition} \mbox{\bf{(Kedlaya)}} \label{proposition5.11con}
Assume now the Banach algebras are commutative. Consider the following exact sequence of Banach algebras satisfying the corresponding conditions in \cite[Definition 2.7.3 (a),(b)]{KL1}:
\begin{align}
0\rightarrow {\Pi}\rightarrow \Pi_1\bigoplus \Pi_2\rightarrow \Pi_{12}\rightarrow 0.	
\end{align}
And consider the corresponding glueing datum $(M_1,M_2,M_{12})$ over the corresponding rings above, and we assume that the glueing datum is finite projective. Then we have that the corresponding kernel $M$ of
\begin{align}
M_1\bigoplus M_2\rightarrow M_{12}	
\end{align}
as module over $\Pi$ is finite projective, with the corresponding isomorphisms:
\begin{align}
M\otimes \Pi_1 \overset{\sim}{\rightarrow}	M_1,\\
M\otimes \Pi_2 \overset{\sim}{\rightarrow}	M_2.
\end{align}

\end{proposition}

\begin{proof}
This proof is due to Kedlaya. The corresponding finitely presentedness is given in \cite[Theorem 1.3.9(a)]{KL1}, note here that we are under the conditions (a), (b) in \cite[Definition 2.7.3 (a),(b)]{KL1}. To promote the corresponding property to that being finite projective. We look at the corresponding diagram coming from the desired presentation:
\[
\xymatrix@R+3pc@C+3pc{
B \ar[r] \ar[r] \ar[r] \ar[d] \ar[d] \ar[d] &B_1\oplus B_2 \ar[r] \ar[r] \ar[r] \ar[d] \ar[d] \ar[d] &B_{12} \ar[d] \ar[d] \ar[d]\\
A \ar[r] \ar[r] \ar[r] \ar[d] \ar[d] \ar[d] &A_1\oplus A_2 \ar[r] \ar[r] \ar[r] \ar[d] \ar[d] \ar[d] &A_{12} \ar[d] \ar[d] \ar[d]\\
M \ar[r] \ar[r] \ar[r]  &M_1\oplus M_2 \ar[r] \ar[r] \ar[r]  &M_{12}. \\
}
\]
Now we look at the sets:
\begin{align}
\mathrm{Hom}(M_1,B_1),\mathrm{Hom}(M_2,B_2)	
\end{align}
which actually are also modules over the corresponding base rings. By \cite[Theorem 1.3.9(a)]{KL1} we can basically apply the construction to 
\begin{align}
\mathrm{Hom}(M_1,B_1),\mathrm{Hom}(M_2,B_2),\mathrm{Hom}(M_{12},B_{12}) 	
\end{align}	
which relates directly to a glueing square, which gives rise to the surjective morphism:
\begin{align}
\mathrm{Hom}(M_1,B_1)\oplus\mathrm{Hom}(M_2,B_2)\rightarrow \mathrm{Hom}(M_{12},B_{12}).	
\end{align}	
Then we choose the corresponding splittings $s_1,s_2$ for $M_1$ and $M_2$ respectively. Then map them to $\mathrm{Hom}(M_{12},B_{12})$. Then we have there is an element $s\in \mathrm{Hom}(M_{12},B_{12})$ such that $f_{1,12}(s_1)-f_{2,12}(s_2)=s$. Then choose the corresponding preimage $s'_1+s'_2$ of $s$ in $\mathrm{Hom}(M_1,B_1)\oplus\mathrm{Hom}(M_2,B_2)$ through:
\begin{align}
\mathrm{Hom}(M_1,B_1)\oplus\mathrm{Hom}(M_2,B_2)\rightarrow \mathrm{Hom}(M_{12},B_{12}).	
\end{align}
This makes the corresponding correct modification to the corresponding splitting $s_1+s_2$ to make sure that $s_1''+s_2''$ lives in the kernel of 
\begin{align}
\mathrm{Hom}(M_1,B_1)\oplus\mathrm{Hom}(M_2,B_2)\rightarrow \mathrm{Hom}(M_{12},B_{12}).	
\end{align}
Then this gives rise to a corresponding splitting for $M$.
\end{proof}

\begin{proposition} \mbox{\bf{(After Kedlaya)}} \label{proposition5.11}
Consider the following exact sequence of Banach algebras satisfying the corresponding conditions in \cite[Definition 2.7.3 (a),(b)]{KL1}:
\begin{align}
0\rightarrow {\Pi}\rightarrow \Pi_1\bigoplus \Pi_2\rightarrow \Pi_{12}\rightarrow 0.	
\end{align}
And consider the corresponding glueing datum $(M_1,M_2,M_{12})$ over the corresponding rings above, and we assume that the glueing datum is finite projective. Then we have that the corresponding kernel $M$ of
\begin{align}
M_1\bigoplus M_2\rightarrow M_{12}	
\end{align}
as bimodule over $\Pi$ is finite projective, with the corresponding isomorphisms:
\begin{align}
M\otimes \Pi_1 \overset{\sim}{\rightarrow}	M_1,\\
M\otimes \Pi_2 \overset{\sim}{\rightarrow}	M_2.
\end{align}

\end{proposition}

\begin{proof}
The left module and the right module structures have already given us the corresponding finitely presentedness. To promote the corresponding property to that being finite projective. We look at the corresponding diagram coming from the desired presentation:
\[
\xymatrix@R+3pc@C+3pc{
B \ar[r] \ar[r] \ar[r] \ar[d] \ar[d] \ar[d] &B_1\oplus B_2 \ar[r] \ar[r] \ar[r] \ar[d] \ar[d] \ar[d] &B_{12} \ar[d] \ar[d] \ar[d]\\
A \ar[r] \ar[r] \ar[r] \ar[d] \ar[d] \ar[d] &A_1\oplus A_2 \ar[r] \ar[r] \ar[r] \ar[d] \ar[d] \ar[d] &A_{12} \ar[d] \ar[d] \ar[d]\\
M \ar[r] \ar[r] \ar[r]  &M_1\oplus M_2 \ar[r] \ar[r] \ar[r]  &M_{12}. \\
}
\]
Now we look at the sets:
\begin{align}
\mathrm{Hom}(M_1,B_1),\mathrm{Hom}(M_2,B_2)	
\end{align}
which actually are also left modules. By the left glueing process we consider before we can basically apply the construction to 
\begin{align}
\mathrm{Hom}(M_1,B_1),\mathrm{Hom}(M_2,B_2),\mathrm{Hom}(M_{12},B_{12}) 	
\end{align}	
which relates directly to a glueing square, which gives rise to the surjective morphism:
\begin{align}
\mathrm{Hom}(M_1,B_1)\oplus\mathrm{Hom}(M_2,B_2)\rightarrow \mathrm{Hom}(M_{12},B_{12}).	
\end{align}	
This is again as left module consideration. Then we choose the corresponding splittings $s_1,s_2$ for $M_1$ and $M_2$ respectively. Then map them to $\mathrm{Hom}(M_{12},B_{12})$. Then we have there is an element $s\in \mathrm{Hom}(M_{12},B_{12})$ such that $f_{1,12}(s_1)-f_{2,12}(s_2)=s$. Then choose the corresponding preimage $s'_1+s'_2$ of $s$ in $\mathrm{Hom}(M_1,B_1)\oplus\mathrm{Hom}(M_2,B_2)$ through:
\begin{align}
\mathrm{Hom}(M_1,B_1)\oplus\mathrm{Hom}(M_2,B_2)\rightarrow \mathrm{Hom}(M_{12},B_{12}).	
\end{align}
This makes the corresponding correct modification to the corresponding splitting $s_1+s_2$ to make sure that $s_1''+s_2''$ lives in the kernel of 
\begin{align}
\mathrm{Hom}(M_1,B_1)\oplus\mathrm{Hom}(M_2,B_2)\rightarrow \mathrm{Hom}(M_{12},B_{12}).	
\end{align}
Then this gives rise to a corresponding splitting for $M$ as left module. Then on the other hand, we look at the corresponding diagram coming from the desired presentation:
\[
\xymatrix@R+3pc@C+3pc{
B \ar[r] \ar[r] \ar[r] \ar[d] \ar[d] \ar[d] &B_1\oplus B_2 \ar[r] \ar[r] \ar[r] \ar[d] \ar[d] \ar[d] &B_{12} \ar[d] \ar[d] \ar[d]\\
A \ar[r] \ar[r] \ar[r] \ar[d] \ar[d] \ar[d] &A_1\oplus A_2 \ar[r] \ar[r] \ar[r] \ar[d] \ar[d] \ar[d] &A_{12} \ar[d] \ar[d] \ar[d]\\
M \ar[r] \ar[r] \ar[r]  &M_1\oplus M_2 \ar[r] \ar[r] \ar[r]  &M_{12}. \\
}
\]
Now we look at the sets:
\begin{align}
\mathrm{Hom}(M_1,B_1),\mathrm{Hom}(M_2,B_2)	
\end{align}
which actually are also right modules. By the right glueing process we consider before we can basically apply the construction to 
\begin{align}
\mathrm{Hom}(M_1,B_1),\mathrm{Hom}(M_2,B_2),\mathrm{Hom}(M_{12},B_{12}) 	
\end{align}	
which relates directly to a glueing square, which gives rise to the surjective morphism:
\begin{align}
\mathrm{Hom}(M_1,B_1)\oplus\mathrm{Hom}(M_2,B_2)\rightarrow \mathrm{Hom}(M_{12},B_{12}).	
\end{align}	
This is again as right module consideration. Then we choose the corresponding splittings $s_1,s_2$ for $M_1$ and $M_2$ respectively. Then map them to $\mathrm{Hom}(M_{12},B_{12})$. Then we have there is an element $s\in \mathrm{Hom}(M_{12},B_{12})$ such that $f_{1,12}(s_1)-f_{2,12}(s_2)=s$. Then choose the corresponding preimage $s'_1+s'_2$ of $s$ in $\mathrm{Hom}(M_1,B_1)\oplus\mathrm{Hom}(M_2,B_2)$ through:
\begin{align}
\mathrm{Hom}(M_1,B_1)\oplus\mathrm{Hom}(M_2,B_2)\rightarrow \mathrm{Hom}(M_{12},B_{12}).	
\end{align}
This makes the corresponding correct modification to the corresponding splitting $s_1+s_2$ to make sure that $s_1''+s_2''$ lives in the kernel of 
\begin{align}
\mathrm{Hom}(M_1,B_1)\oplus\mathrm{Hom}(M_2,B_2)\rightarrow \mathrm{Hom}(M_{12},B_{12}).	
\end{align}
Then this gives rise to a corresponding splitting for $M$ as right module as well.

\end{proof}

\indent Apply the corresponding general results developed above we have the following result for glueing finite projective Frobenius modules.

\begin{proposition} \label{proposition5.13}
Consider the following two categories. The first category is the corresponding finite projective $\varphi^a$-bibundles over the corresponding period ring $\widetilde{\Pi}_{R,B}$. And the corresponding second finite projective $\varphi^a$-bimodules over the period ring $\widetilde{\Pi}^{[s,r]}_{R,B}$	such that we have $0<s\leq rp^{-ha}$. Then we have that the corresponding involved categories are actually equivalent to each other.
\end{proposition}

\begin{proof}
The functor from the first category to the second one realizing this equivalence is just the corresponding projection. To show that for any bimodule over $\widetilde{\Pi}^{[s,r]}_{R,B}$ we have the corresponding essential surjectivity we lift the corresponding bimodule up to achieve a family of bimodules with the same rank by applying the corresponding Frobenius with respect to different intervals taking the form of $[sp^{-akh},rp^{-akh}]$. Then for general interval we are free to glue finite projective modules by applying \cref{proposition5.11}.\\
\end{proof}

\newpage

\subsection*{Acknowledgements} 

We would like to thank Professor Kedlaya for helpful discussion on the descent, which was obviously one of the reasons for us to write this paper. The corresponding ideas of forming mixed-type period rings and the mixed-type Hodge-structures are essentially inspired by the work of Professor Kedlaya (such as in the work around Drinfeld's Lemma of perfectoid spaces). We benefit a lot as well from the discussion with Federico Bambozzi on the work of Bambozzi-Kremnizer around the corresponding $\infty$-sheafiness and $\infty$-analytic spaces.

\newpage

\bibliographystyle{ams}

\begin{thebibliography}{10}
\bibitem[KL1]{KL1} Kedlaya, Kiran Sridhara, and Ruochuan Liu. Relative $p$-adic Hodge theory: foundations. Soci\'et\'e math\'ematique de France, 2015.

\bibitem[KL2]{KL2} Kedlaya, Kiran S., and Ruochuan Liu. "Relative $p$-adic Hodge theory, II: Imperfect period rings." arXiv preprint arXiv:1602.06899 (2016).

\bibitem[T1]{T1} Tong, Xin. "Hodge-Iwasawa Theory I." arXiv preprint arXiv:2006.03692 (2020).
\bibitem[T2]{T2} Tong, Xin. "Hodge-Iwasawa Theory II." arXiv preprint arXiv:2010.06093 (2020).
\bibitem[T3]{T3} Tong, Xin. "Analytic Geometry and Hodge-Frobenius Structure." arXiv preprint arXiv:2011.08358 (2020).
\bibitem[CKZ]{CKZ} Carter, Annie, Kiran S. Kedlaya, and Gergely Z\'abr\'adi. "Drinfeld's lemma for perfectoid spaces and overconvergence of multivariate $(\varphi,\Gamma) $-modules." arXiv preprint arXiv:1808.03964 (2018).
\bibitem[PZ]{PZ} Pal, Aprameyo, and Gergely Z\'abr\'adi. "Cohomology and overconvergence for representations of powers of Galois groups." arXiv preprint arXiv:1705.03786 (2017).


\bibitem[KL]{KL} Kedlaya, Kiran, and Ruochuan Liu. "On families of $(\varphi, \Gamma)$-modules." Algebra and Number Theory 4, no. 7 (2011): 943-967.
\bibitem[Lu1]{Lu1} Lurie, Jacob. "Higher algebra, preprint 2014." (2017).

\bibitem[Lu2]{Lu2} Lurie, Jacob. Higher topos theory. Princeton University Press, 2009.


\bibitem[BK]{BK} Bambozzi, Federico, and Kobi Kremnizer. "On the Sheafyness Property of Spectra of Banach Rings." arXiv preprint arXiv:2009.13926 (2020).


\bibitem[BBBK]{BBBK} Bambozzi, Federico, Oren Ben-Bassat, and Kobi Kremnizer. "Analytic geometry over $\mathbb{F}_1$ and the Fargues-Fontaine curve." Advances in Mathematics 356 (2019): 106815.


\bibitem[CS]{CS} Clausen, Dustin. Peter Scholze. "Lectures on analytic geometry." Notes available at http://www.math.uni-bonn.de/people/scholze/Notes.html. 


\bibitem[Sch1]{Sch1} Scholze, P. "\'Etale cohomology of diamonds, Preprint (2017)." arXiv preprint arXiv:1709.07343.



\end{thebibliography}

\end{document}